\pdfoutput=1

\documentclass[11pt]{article}

\usepackage[margin=1in]{geometry}

\PassOptionsToPackage{svgnames,dvipsnames,table}{xcolor}

\usepackage{biblatex}
\usepackage{xcolor}
\usepackage{tikz}
\usetikzlibrary{positioning}
\usepackage{tikz-cd}

\usepackage{amsmath}
\usepackage{amssymb}
\usepackage{amsthm}
\usepackage{thmtools}
\usepackage{mathtools}
\usepackage{mathdots}
\usepackage[minimal]{yhmath}
\usepackage{derivative}
\usepackage{stmaryrd}

\usepackage{bm}

\usepackage{mathrsfs}
\usepackage{graphicx}
\usepackage{url}
\usepackage{float}

\usepackage{array}
\usepackage{longtable}
\usepackage{epigraph}
\usepackage{listings}
\usepackage{genyoungtabtikz}
\usepackage{caption}
\usepackage{subcaption}

\input xy
\xyoption{all}

\usepackage{hyperref}
\allowdisplaybreaks

\theoremstyle{plain}
\newtheorem{theorem}{Theorem}[section]
\newtheorem{proposition}[theorem]{Proposition}
\newtheorem{lemma}[theorem]{Lemma}

\theoremstyle{definition}
\newtheorem{definition}[theorem]{Definition}
\newtheorem{example}[theorem]{Example}

\theoremstyle{remark}
\newtheorem*{remark}{Remark}

\newcommand{\floor}[1]{\left\lfloor #1 \right\rfloor}

\newcommand{\eps}{\epsilon}

\newcommand{\CC}{\mathbb C}

\newcommand{\NN}{\mathbb N}
\newcommand{\QQ}{\mathbb Q}

\newcommand{\ZZ}{\mathbb Z}
\newcommand{\PP}{\mathbb P}

\newcommand{\opname}{\operatorname}

\DeclareMathOperator{\Aut}{Aut}

\renewcommand{\Im}{\operatorname{Im}}

\DeclareMathOperator{\Hom}{Hom}

\DeclareMathOperator{\End}{End}

\DeclareMathOperator{\Lie}{Lie}

\newcommand{\mc}{\mathcal}

\newcommand{\mf}{\mathfrak}
\newcommand{\msf}{\mathsf}

\DeclareMathOperator{\Rep}{\operatorname{Rep}}

\newcommand{\taking}[1]{\xrightarrow{#1}}

\newcommand{\bbone}{\text{\usefont{U}{bbold}{m}{n}1}}
\MakeRobust{\bbone}

\newcommand{\bs}{\boldsymbol}
\newcommand{\ev}{\opname{ev}}

\newcommand{\msV}{\mathsf{V}}

\newcommand{\tb}{\mathscr{V}}
\newcommand{\pst}{\mathcal{P}}
\DeclareMathOperator{\rpp}{rpp}
\DeclareMathOperator{\col}{c}
\DeclareMathOperator{\weight}{wt}
\newcommand{\pol}{T^{1/2}}

\newcommand*\circled[1]{\tikz[baseline=(char.base)]{
            \node[shape=circle,draw,inner sep=1.5pt,minimum height=18pt,minimum width=18pt] (char) {#1};}}
            
\newcommand*\squared[1]{\tikz[baseline=(char.base)]{
            \node[shape=rectangle,draw,inner sep=1.5pt,minimum height=18pt,minimum width=18pt] (char) {#1};}}

\title{Vertex functions of type $D$ Nakajima quiver varieties}

\author{Hunter Dinkins and Jiwu Jang}

\begin{document}







\maketitle

\abstract{
    \quad We study the quasimap vertex functions of type $D$ Nakajima quiver varieties that have framings only at minuscule nodes. We compute the coefficients of the vertex functions in the $K$-theoretic fixed point basis. We also give an explicit combinatorial description of zero-dimensional type $D$ quiver varieties and their vertex functions using the combinatorics of minuscule posets. 
    
    \quad We prove that these zero-dimensional vertex functions can be expressed as products of $q$-binomial functions. This proves a degeneration of the conjectured 3d mirror symmetry of vertex functions which allows one to extract the character of the tangent space of a Coulomb branch from the vertex function of the corresponding Higgs branch. The proof relies on the theory of Macdonald polynomials and proceeds by iterated applications of the Cauchy identity. As a by-product of the proof, we provide an interpretation of type $D$ spin vertex functions as the partition functions of the half-space Macdonald processes of Barraquand, Borodin, and Corwin. This hints that the geometry of quiver varieties may provide new examples of integrable probabilistic models.
}

\setcounter{tocdepth}{1}
\tableofcontents

\section{Introduction}\label{sec:introduction}

\subsection{Quiver varieties and 3d mirror symmetry}

This paper is about vertex functions of type $D$ Nakajima quiver varieties. Nakajima quiver varieties are one of the central objects in geometric representation theory, defined by Nakajima in \cite{nakajima1994instantons} to give a geometric construction of representations of Kac--Moody algebras. They include as special cases cotangent bundles of partial flag varieties and moduli spaces of sheaves on resolutions of Kleinian singularities. They are also one of the main ways of constructing examples of symplectic resolutions, and hence fit into the programs of symplectic duality and 3d mirror symmetry studied, for example, in \cite{aganagic2021elliptic, blpw2014quantizations, dinkins20203dflag, dinkins2022symplectic, kamnitzer2022symplectic, kononov22pursuing2, intriligatorseiberg1996mirror, rimanyi2021mirror}.

Vertex functions, defined by Okounkov in \cite{okounkov2017lectures}, are generating functions for certain counts of quasimaps to quiver varieties. They are $K$-theoretic analogs of $I$-functions from Gromov--Witten theory. They provide interesting examples of $q$-series, containing so-called basic hypergeometric series \cite{gasperrahman2004}, and are central in 3d mirror symmetry.

\subsection{Vertex functions in type \texorpdfstring{$D$}{D}}

In the existing literature, most results about vertex functions are split between those proven for arbitrary quiver varieties, such as \cite{aganagicokounkov2017quasimap, Pushk2, okounkov2017lectures}, and those relying on explicit computations in the type $A$ or hypertoric settings, for example \cite{dinkins2024vertex, dinkinssmirnov2020quasimaps, Liuqmstab,  tRSKor, KorZeit2, psz2020baxterquantum, smirnovzhou2022hypertoric}. The main goal of this paper is to produce explicit computations of vertex functions for $D$ quiver varieties. One immediate complication is that the set of torus fixed points on these varieties need not be zero-dimensional, as it is in the type $A$ or hypertoric settings. Indeed, the geometry of the fixed components of type $D$ quiver varieties need not be trivial, and their Betti numbers were described by Nakajima in \cite{nakajima2004quiver}.

To deal with this, we will restrict ourselves to quiver varieties with framing vectors which are nonzero only at minuscule vertices, for which this issue does not arise\footnote{In general, the moduli space of quasimaps used to define vertex functions can have nontrivial torus fixed components, even if the quiver variety has finite fixed points. A consequence of this paper is that the restriction to minuscule framings is sufficient (though not strictly necessary) to guarantee that torus fixed quasimaps will be isolated as well.}. On such quiver varieties, we will give an explicit combinatorial description of the torus fixed points in terms of the so-called minuscule posets (see \cite{proctor1999dynkin, proctor1986dynkinclassification, stembridge1994minuscule}); see Theorems~\ref{posetthm1} and~\ref{posetthm2}. For the case of a one-dimensional framing where the quiver varieties parameterize certain modules over the preprojective algebra of the quiver, this relationship was noticed before in \cite{dekmf2024heaps}. From this description, we can immediately read off the weights of the restrictions of the tautological bundles to fixed points; see Equation~\eqref{tbweights}. This poset description also turns out to be enough to allow us to use the localization theorem to write explicit formulas for these type $D$ vertex functions.

\begin{theorem}[Theorem~\ref{verformula}]
    There is an explicit combinatorial formula for vertex functions of type $D$ quiver varieties with framings at minuscule vertices.
\end{theorem}

\subsection{Zero-dimensional type \texorpdfstring{$D$}{D} varieties}

As proven in \cite{okounkov2017lectures}, when one sends equivariant parameters to some infinity in the vertex function of a quiver variety, one obtains the vertex function of the fixed point set. Such vertex functions are nontrivial even if the fixed points are isolated. From a sufficiently deep understanding of them, one can read off formulas for the characters of the tangent space of the 3d mirror dual variety; see \cite{dinkinssmirnov2020characters} for more details. Vertex functions of zero-dimensional type $A$ quiver varieties were studied in \cite{dinkinssmirnov2020quasimaps}, where it was shown that such varieties are in natural bijection with Young diagrams. Certain product formulas, which are the type $A$ version of our Theorem~\ref{thmintro1} below, were also proven.

For our goals in this paper, it is thus important to study the (nontrivial) vertex functions for zero-dimensional type $D$ quiver varieties. We turn our attention to this in Section~\ref{sec:main-thm}, proving a formula expressing the vertex functions of these zero-dimensional type $D$ quiver varieties as products of $q$-binomial series.

To state the formula, we need some notation. We fix a type $D_{n}$ quiver and denote by $\Phi$ the set of roots. Let $\alpha_{i}$ for $1 \leq i \leq n$ be a choice of simple roots, $\omega_{i}$ be the fundamental weights, and $\Phi^{+}$ be the set of positive roots. Let $\msf{v}$ and $\msf{w}$ be the dimension and framing vectors, respectively, such that $\msf{w}_{i}=1$ for a single minuscule vertex $i$ and is $0$ otherwise. Let $\mc{M}(\msf{v},\msf{w})$ be the Nakajima quiver variety associated to this data. As usual, it should be thought of as associated to the space of weight $\mu:=\sum_{i} \msf{w}_{i} \omega_{i} - \sum_{i} \msf{v}_{i} \alpha_{i}$ inside a representation with highest weight $\lambda:=\sum_{i} \msf{w}_{i} \omega_{i}$. The vertex function $V(\bs z)$ of $\mc{M}(\msf{v},\msf{w})$ is an element of the ring $\mathbb{Q}(q,\hbar)[[z_{i}]]_{1 \leq i \leq n}$. We identify this ring with (a completion of) the group algebra of positive roots by the assignment $e^{\alpha_{i}}=\left(q/\hbar\right)^{(\msf{v}-A_{Q}\msf{v})_{i}}z_{i}$ where $A_{Q}$ is the adjacency matrix of the quiver.

\begin{theorem}[Theorem~\ref{thm:product-identity-d}]\label{thmintro1}
    Let $\Phi^{+}_{\mu} = \{\alpha \in \Phi^{+} \mid (\mu, \alpha) < 0\}$. The vertex function $\msV(\bs z)$ of $\mc{M}(\msf v, \msf w)$ factorizes into the following product of $q$-binomial functions:
    \[
        \msV(\bs z) = \prod_{\alpha \in \Phi^{+}_{\mu}} F(e^{\alpha}),
    \]
\end{theorem}
In the previous theorem, $F$ is the $q$-binomial function, defined by
\[
    F(z):=\prod_{i=0}^{\infty} \frac{1-\hbar z q^i}{1-z q^i} 
\]

Our method of proof is the following. First, we start with the localization formula for the vertex function, Theorem~\ref{verformula}, which works for type $D$ quiver varieties that are not necessarily single points. For quiver varieties which are single points, and for these only, we are able to rewrite the localization formula as a certain sum of skew-Macdonald polynomials. Finally, we compute this sum by repeated applications of the skew-Cauchy identity.

\subsection{Consequences}

We mention a few applications of Theorem~\ref{thmintro1}. As discussed above, one is supposed to be able to read off from this the character of the tangent space of the 3d mirror dual variety. By comparing with the formulas of \cite{krylovperunov2021almost}, Theorem~\ref{thmintro1} shows immediately that this is the case. Vertex functions for arbitrary quiver varieties are also conjectured to have a certain relation to the vertex functions of their 3d mirror dual varieties \cite{aganagic2021elliptic, bottadink, dinkins20203dflag, dinkins2022symplectic, smirnovzhou2022hypertoric}. All of the cases in which this has been proven use formulas like Theorem~\ref{thmintro1} as an ingredient. We hope that this will one day be the case for Theorem~\ref{thmintro1} as well.

The techniques used here also allow us to study vertex functions with descendants. As in \cite{dinkinssmirnov2020capped}, descendants given by exterior powers of tautological vector bundles can be inserted by applying a Macdonald difference operator to the vertex with trivial descendant. More generally, any $q$-difference operator acting diagonally on Macdonald polynomials with eigenvalues given as a function of $\hbar^{n-i} q^{\lambda_{i}}$ (for example, Noumi's $q$-difference operator from \cite{noumisano2021ruijsenaars}), can be used to insert descendants. As in \cite{dinkinssmirnov2020capped}, one can interpret this as giving an explicit rational formula for the \emph{capped} descendant vertex, an object significantly harder to compute than the (bare) vertex.

When the framing $\msf{w}$ is at a spin node, Theorem~\ref{thm:product-identity-d} also has a probabilistic interpretation in terms of Macdonald processes. Macdonald processes were first introduced by A.~Borodin and I.~Corwin in \cite{borodincorwin2014macdonald}. They are probability measures on sequences of partitions defined in terms of Macdonald symmetric functions and depending on two parameters $q$ and $\hbar$. \emph{Half-space} Macdonald processes were introduced by Barraquand, Borodin, and Corwin in \cite{bbc2020halfspacemacdonald}, and have connections to interesting half-space systems, such as the Kardar--Parisi--Zhang (KPZ) stochastic PDE (see \cite{kpz1986dynamic, corwin2012kardar}) and the log-gamma directed polymer in a half-quadrant (see \cite{oconnell2013geometricrsk}).

In the case of type $D_{n}$ quiver varieties with framing $\msf{w}=(0,\ldots,0,1)$, Theorem~\ref{thmintro1} shows that vertex functions are exactly partition functions for a half-space Macdonald process. The type $A$ version of this statement, relating vertex functions to ordinary Macdonald processes is a consequence of \cite{dinkinssmirnov2020quasimaps}, though it was not noticed there. From this perspective, vertex functions with descendants compute the expectation values of certain observables of (ordinary or half-space) Macdonald processes.

\subsection{Future directions}

The results of this paper point to several future directions.

One such direction is to complete the study of simply laced Dynkin diagrams by studying vertex functions of type $E$ quiver varieties where the framing is a sum of minuscule framings. Indeed, the description of fixed points in terms of minuscule posets and the subsequent localization formula for vertex functions, applies immediately to this setting. We hope that the symmetric function techniques of this paper should be applicable to that setting as well, leading to product formulas for the vertex functions of zero-dimensional type $E$ quiver varieties.

There are also quiver varieties which are not Dynkin type which still have some zero-dimensional torus fixed components. We believe the methods of this paper can be extended to cover some of these situations as well, namely, those in which the torus fixed points can be described as slant sums, see \cite{proctor1999dynkin}, of minuscule posets.

Finally, given the correspondence between the partition function of Macdonald processes and half-space Macdonald processes with vertex functions of type $A$ and $D$ quiver varieties, we hope that the geometry of quiver varieties will give rise to new and interesting probabilistic processes. According to this dream, this paper already provides something new: vertex functions of type $D$ quiver varieties with framing $\msf{w}=(1,0,\ldots,0)$ should be interpreted as the partition function of a new probabilistic process. An expression for this partition function in terms of symmetric polynomials can be found in Lemma \ref{verequalsmac}. Whether this process is related to other interesting probabilistic models like in the spin case is at present unknown to us. We hope to revisit this in the future.

\subsection{Outline}
This paper is structured as follows.

In Section~\ref{sec:background}, we give background on Nakajima quiver varieties.

In Section~\ref{sec:macdonald}, we review the properties of Macdonald polynomials that we require.

In Section~\ref{sec:stable-quiver-rep}, we give an explicit description of zero-dimensional type $D$ quiver varieties with a single framing at a minuscule vertex. We formulate this in terms of the combinatorics of minuscule posets and explain how to read off weights of tautological vector bundles from this description.

In Section~\ref{sec:type-d-vertex-functions}, we review the definition of vertex functions. Then we spell out the localization formula in the type $D$ setting.

In Section~\ref{sec:main-thm}, we state our product formula for zero-dimensional type $D$ quiver varieties and explain some applications.

In Sections~\ref{sec:thm-proof1} and~\ref{sec:thm-proof2} we prove the product formula using the localization theorem and properties of Macdonald polynomials.

\section*{Acknowledgments}

The authors would like to thank Alexei~Borodin for providing significant insight for proof of Theorem~\ref{thm:product-identity-d}. This project also benefited from correspondence with Vasily~Krylov, Elijah~Bodish, and Hiraku~Nakajima. The authors would like to extend this gratitude to the PRIMES-USA program at MIT for the research opportunity and hospitality. H.~\!Dinkins was supported by National Science Foundation grant DMS-2303286 at MIT and the National Science Foundation Research Training Group grant Algebraic Geometry and Representation Theory at Northeastern University DMS-1645877.

\section{Background}\label{sec:background}

The goal of this section is to review the definition of Nakajima quiver varieties, defined in \cite{nakajima1998quiverkacmoody, nakajima1994instantons}, which are moduli spaces of certain quiver representations. The following material is standard, and can be found in \cite{ginzburg2009lectures}, for example.

\subsection{Definition of Nakajima quiver varieties}\label{subsec:nakajima}

Let $Q = (Q_0, Q_1)$ be a quiver with vertex set $Q_0$ and arrow set $Q_1$.  Let $o(e)$ and $i(e) \in Q_0$ denote the outgoing (source) and incoming (target) vertices of the arrow $e \in Q_1$, respectively, so that $e$ is an arrow from $o(e)$ to $i(e)$. Equivalently, this is written as $o(e) \taking{e} i(e) \in Q_1$.

Choose $\msf{v},\msf{w} \in \mathbb{N}^{Q_0}$. Let $V_{i}$ and $W_{i}$ be complex vector spaces such that $\dim V_{i}=\msf{v}_{i}$ and $\dim W_{i}=\msf{w}_{i}$. Let
\[
    \Rep_{Q}(\msf{v},\msf{w}):=\bigoplus_{e \in Q_1}\Hom\left(V_{o(e)},V_{i(e)}\right) \oplus \bigoplus_{i \in Q_0} \Hom\left(W_{i},V_{i}\right).
\]
This is known as the space of framed representations of $Q$. The vector $\msf{v}$ is called the dimension, and $\msf{w}$ is called the framing dimension. The trace pairing identifies $\Hom(V_{i},V_{j})^{*}$ with $\Hom(V_{j},V_{i})$ and thus induces an isomorphism
\begin{equation}
    \begin{split}\label{T*rep}
        T^{*} \Rep_{Q}(\msf{v},\msf{w}) & \cong \bigoplus_{e \in Q_1}\Hom\left(V_{o(e)},V_{i(e)}\right) \oplus \Hom\left(V_{i(e)},V_{o(e)}\right) \\
                                        & \qquad\oplus \bigoplus_{i \in Q_0} \Hom\left(W_{i},V_{i}\right) \oplus \Hom\left(V_{i},W_{i}\right).
    \end{split}
\end{equation}
The group $G_{\msf{v}}:=\prod_{i \in Q_0} GL(V_{i})$ acts on $\Rep_{Q}(\msf{v},\msf{w})$, inducing a Hamiltonian action on $T^{*} \Rep_{Q}(\msf{v},\msf{w})$ with associated moment map
\[
    \mu: T^{*} \Rep_{Q}(\msf{v},\msf{w}) \to \Lie(G_{\msf{v}})^{*} \cong \prod_{i \in Q_0} \End(V_{i}),
\]
where the isomorphism is by the trace pairing. Let $\theta$ be a character of $G_{\msf{v}}$, which we will view as a vector $\theta \in \mathbb{Z}^{Q_0}$ with associated character
\begin{align*}
    \chi_{\theta}: G_{\msf{v}} & \to \mathbb{C}^{\times},                            \\
    (g_{i})_{i \in Q_0}        & \mapsto \prod_{i \in Q_0} \det(g_{i})^{\theta_{i}}.
\end{align*}

\begin{definition}[\cite{nakajima1994instantons}]\label{nqv}
    The Nakajima quiver variety associated to the data $Q,\msf{v},\msf{w},$ and $\theta$ is the algebraic symplectic reduction
    \[
        \mc{M}_{Q,\theta}(\msf{v}, \msf{w}) :=T^{*}\Rep_{Q}(\msf{v},\msf{w})\,/\!\!/\!\!/\!\!/_{\!\theta}\, G_{\msf{v}}=\mu^{-1}(0)\,/\!\!/_{\!\theta}\,G_{\msf{v}}.
    \]
\end{definition}
The notation $\mu^{-1}(0)\,/\!\!/_{\!\theta}\,G_{\msf{v}}$ refers to the GIT quotient of $\mu^{-1}(0)$ by $G_{\msf{v}}$ with respect to the stability parameter $\theta$.

In this paper, we will only consider the positive stability parameter $\theta=(1,1,\dots,1)$. So from now on, we will simplify notation and write $\mc{M}_{Q}(\msf{v}, \msf{w})$ with this choice of stability understood.

\subsection{More explicit description}
We will now make Definition~\ref{nqv} more explicit.

Points of $\mc{M}_{Q}(\msf{v}, \msf{w})$ are represented by elements of $T^*\Rep_{Q}(\msf{v},\msf{w})$. Using \eqref{T*rep}, we will denote elements of this latter space by a tuple
\[
    \left(\left(X_e\right)_{e \in Q_1},\left(Y_e\right)_{e \in Q_1},\left(A_i\right)_{i \in Q_0},\left(B_i\right)_{i \in Q_0}\right),
\]
where $X_e \in \Hom(V_{o(e)}, V_{i(e)})$, $Y_e \in \Hom(V_{i(e)}, V_{o(e)})$, $A_i \in \Hom(W_i, V_i)$, and $B_i \in \Hom(V_i, W_i)$. We will sometimes abuse notation and denote such a tuple simply by $(X,Y,A,B)$.
In this notation, the moment map is given explicitly by
\begin{equation}\label{eq:moment-map}
    \begin{split}
         & \mu\left(\left(\left(X_e\right)_{e \in Q_1},\left(Y_e\right)_{e \in Q_1},\left(A_i\right)_{i \in Q_0},\left(B_i\right)_{i \in Q_0}\right)\right) \\
         & = \left(\sum_{\substack{e \in Q_1                                                                                                                \\ i(e)=i}} X_e Y_e-\sum_{\substack{e \in Q_1 \\ o(e)=i}} Y_e X_e+A_i B_i\right)_{i \in Q_0}.
    \end{split}
\end{equation}
For generic choices of $\theta$, including the positive stability, the notions of $\theta$-stability and $\theta$-semistability from geometric invariant theory coincide; see Section~3 of \cite{ginzburg2009lectures}. Furthermore, it is known that $G_{\msf{v}}$ acts freely on the set of stable points. Thus $\mc{M}_{Q}(\msf{v}, \msf{w})$ is a smooth quasiprojective variety which as a set, is equal to
\[
    \mu^{-1}(0)^{s}/G_{\msf{v}},
\]
where $\mu^{-1}(0)^{s}$ denotes the set of stable points in $\mu^{-1}(0)$. The set of stable points in $T^*\Rep_{Q}(\msf{v},\msf{w})$ has the following explicit description.

\begin{proposition}[Proposition~5.1.5 of \cite{ginzburg2009lectures}]\label{stability}
    A point $(X, Y, A, B) \in T^{*}\Rep_{Q}(\msf{v},\msf{w})$ is stable, if for any collection $\left(S_i\right)_{i \in Q_0}$, the conditions
    \begin{itemize}
        \item $S_i$ is a subspace of $V_i$,
        \item $\Im(A_i) \subset S_i$,
        \item for all $i \taking{e} j \in Q_1$, we have $X_e(S_i) \subset S_j$ and $Y_e(S_j) \subset S_i$,
    \end{itemize}
    imply that $S_i = V_i$.
\end{proposition}

\subsection{Torus action and fixed points}

The action of $G_{\msf{w}}:=\prod_{i \in Q_0} GL(W_{i})$ on $T^{*}\Rep_{Q}(\msf{v},\msf{w})$ descends to an action on $\mc{M}_{Q}(\msf{v}, \msf{w})$. In addition, the dilation of cotangent fibres also descends to an action of a torus which we will write as $\mathbb{C}^{\times}_{\hbar}$. Choose a decomposition $W_{i} \cong W_{i}' \oplus W_{i}''$ for each $ i \in Q_0$. Let $\msf{w}=\msf{w}'+\msf{w}''$ be the corresponding decomposition of the framing vector. Let $\msf{A}\subset G_{\msf{w}}$ be the rank $1$ torus which acts on each $W_{i}'$ with weight $1$ and on each $W_{i}''$ with weight $0$. Then
\begin{equation}\label{tens}
    \mc{M}_{Q}(\msf{v}, \msf{w})^{\msf{A}}=\bigsqcup_{\substack{\msf{v}',\msf{v}'' \in \mathbb{N}^{Q_0} \\ \msf{v}'+\msf{v}''=\msf{v}}}\mc{M}_{Q}(\msf{v}', \msf{w}')\times \mc{M}_{Q}(\msf{v}'', \msf{w}'').
\end{equation}
This is commonly referred to as the tensor product property of Nakajima varieties \cite{nakajima2001quivertensor}.
Iterating the tensor product property, one sees that fixed points of a maximal torus of $G_{\msf{w}}$ are given by products of Nakajima varieties with one dimensional framing. For $i \in Q_0$, let $\delta_{i}\in \NN^{Q_0}$ be the framing vector with $1$ in position $i$ and $0$ elsewhere.

\begin{proposition}[\cite{nakajima2004quiver}]\label{minframing}
    If $Q$ is an ADE quiver, then \[ \bigsqcup_{\msf{v} \in \NN^{Q_0}} \mc{M}_{Q}(\msf{v},\delta_{i}) \] consists of isolated points if and only if $i$ is a minuscule vertex.
\end{proposition}

\subsection{Polarization}

The vector spaces $V_{i}$ descend to vector bundles $\tb_{i}$ on $\mc{M}$ which are called the tautological vector bundles. Similarly, the spaces $W_{i}$ descend to topologically trivial bundles $\mathscr{W}_{i}$. These bundles carry $\msf{T}$-equivariant structures. Furthermore, the $K$-theory class of the tangent bundle of $\mc{M}$ decomposes as
\[
    T \mc{M}= \pol + \hbar^{-1} (\pol)^{\vee} \in K_{\msf{T}}(\mc{M}),
\]
where
\begin{equation}\label{polarization}
    \pol= \sum_{e \in Q_{1}} \Hom\left(\tb_{o(e)},\tb_{i(e)}\right)+\sum_{i \in Q_{0}} \Hom\left(\mathscr{W}_{i},\tb_{i}\right)-\sum_{i \in Q_{0}} \Hom\left(\tb_{i},\tb_{i}\right).
\end{equation}

\section{Macdonald polynomials}\label{sec:macdonald}
In this section, we review some properties of Macdonald polynomials that we will need later on. The reader is encouraged to skip this section and return to it only as a reference for results used in Sections \ref{sec:thm-proof1} and \ref{sec:thm-proof2}. The following is standard and can be found in \cite{macdonald1995symmetric, gasper1995lecturenotesqseries}. The standard notation in the literature is related to ours by $\hbar=t$.

\subsection{Symmetric functions}
Let $\mathcal{F} = \mathbb{C}(q,\hbar) \otimes_\mathbb{C} \mathbb{C}[p_1, p_2, \dots]$ be the ring of symmetric functions in infinitely many variables $x = (x_1, x_2, \dots)$, where $p_i$ denotes the $i$-th power-sum symmetric function. For a partition $\lambda = (\lambda_1, \lambda_2, \dots)$, we define $p_\lambda = p_{\lambda_1} p_{\lambda_2} \cdots$.
The Macdonald inner product on $\mathcal{F}$ is given by
\begin{equation}\label{eq:mac-inner}
    \langle p_\lambda, p_\mu \rangle = \delta_{\lambda\mu} \prod_{n \geq 1} n^{m_n(\lambda)} m_n(\lambda)! \prod_{i=1}^{l(\lambda)} \frac{1 - q^{\lambda_i}}{1 - \hbar^{\lambda_i}}.
\end{equation}
where $m_n(\lambda)$ is the multiplicity of $n$ in $\lambda$ and $l(\lambda)$ is the number of nonzero parts of $\lambda$.

\subsection{Macdonald polynomials}

Macdonald $P$-polynomials $P_\lambda(x; q, \hbar)$ are uniquely determined by two properties. First, they expand upper-triangularly in the monomial basis $m_\mu$ as
\[
    P_\lambda = \sum_{\mu \trianglelefteq \lambda} u_{\lambda\mu} m_\mu, \quad u_{\lambda\mu} \in \mathbb{Q}(q,\hbar),
\]
where $u_{\lambda\lambda} = 1$ and $\trianglelefteq$ denotes the dominance order on partitions. Second, they are orthogonal with respect to the Macdonald inner product, meaning that $\langle P_\lambda, P_\mu \rangle = 0$ for $\lambda \neq \mu$.
The dual basis $Q_\lambda(x; q, \hbar)$ is uniquely defined by the $(q,\hbar)$-Hall pairing $\langle P_\lambda, Q_\mu \rangle = \delta_{\lambda\mu}$. Specifically, $Q_\lambda = b_\lambda^{-1} P_\lambda$, where
\begin{equation}\label{macnorm}
    b_\lambda = \prod_{\Box \in \lambda} \frac{1 - q^{a(\Box)} \hbar^{l(\Box)+1}}{1 - q^{a(\Box)+1} \hbar^{l(\Box)}},
\end{equation}
and $a(\Box)$ and $l(\Box)$ are the arm and leg lengths of the box $\Box$ in the Young diagram of $\lambda$, defined by
\[
    a(\Box) = \lambda_i - j,
\]
\[
    l(\Box) = \lambda'_j - i.
\]

Let $x$ and $y$ be two independent variable sets. The Macdonald polynomials satisfy the Cauchy identity, which reads
\begin{equation}\label{eq:cauchy}
    \sum_{\lambda} P_\lambda(\bs x) Q_\lambda(\bs y) = \prod_{i,j} \frac{(\hbar x_i y_j;q)_\infty}{(x_i y_j;q)_\infty}=:\Pi(\bs x, \bs y),
\end{equation}
where the product runs over all pairs of indices $i, j$.

\subsection{Skew-Macdonald polynomials}

The skew-Macdonald polynomials $P_{\lambda/\mu}$ and $Q_{\lambda/\mu}$ extend the $P$ and $Q$-polynomials to skew shapes $\lambda/\mu$. They are characterized by
\[
    \langle P_{\lambda/\mu}, Q_\nu \rangle = \langle P_\lambda, Q_\mu Q_\nu \rangle,
\]
\[
    \langle Q_{\lambda/\mu}, P_\nu \rangle = \langle Q_\lambda, P_\mu P_\nu \rangle,
\]
for all partitions $\nu$. If $\mu \not\trianglelefteq \lambda$, then $P_{\lambda/\mu} = Q_{\lambda/\mu} = 0$. The even-leg weight is defined as
\[
    b_\lambda^{\mathrm{el}} = \prod_{\substack{\Box \in \lambda \\ l(\Box) \text{ even}}} b_\lambda(\Box).
\]

\subsection{Pieri rule}
Let $\varphi_{\lambda/\mu}$ and $\psi_{\lambda/\mu}$ be defined by
\begin{align}
    \nonumber \varphi_{\lambda/\mu} & = \prod_{s \in C_{\lambda/\mu}} \frac{b_\lambda(s)}{b_\mu(s)}, \\
    \psi_{\lambda/\mu} & = \prod_{s \in R_{\lambda/\mu} - C_{\lambda/\mu}} \frac{b_\mu(s)}{b_\lambda(s)}, \label{phipsi}
\end{align}
where $C_{\lambda/\mu}$ and $R_{\lambda/\mu}$ denote the sets of columns and rows of $\lambda$ that contain boxes in the skew shape $\lambda/\mu$, respectively.

The Pieri rule for Macdonald polynomials states that
\begin{equation}\label{pieri}
    P_\mu Q_{(a)} = \sum_{\lambda} \varphi_{\lambda/\mu} P_\lambda,
\end{equation}
where the sum is over partitions $\lambda$ such that $\lambda \succ \mu$ and $|\lambda| = |\mu| + a$. Here, $\varphi_{\lambda/\mu}$ are coefficients determined by the skew shape $\lambda/\mu$, and the notation $\lambda \succ \mu$ means that $\lambda$ interlaces $\mu$ from above, meaning that
\[
    \lambda_1 \geq \mu_1 \geq \lambda_2 \geq \mu_2 \geq \dots.
\]

\subsection{One variable specialization}
For a single variable \(x\) and a single-part partition \((r)\), the specialization of Macdonald polynomials is given by
\begin{equation}\label{onevar}
    P_{(r)}(x; q, \hbar) = x^r,\quad
    Q_{(r)}(x; q, \hbar) = \frac{(\hbar)_r}{(q)_r} x^r,
\end{equation}
where $(a)_r = \prod_{k=0}^{r-1} (1 - a q^k)$. Moreover, $P$ and $Q$ vanish when the number of variables is less than the length of the partition:
\[
    P_\lambda(x_1, \dots, x_n; q, \hbar) = 0 \quad \text{if} \quad n < l(\lambda),
\]
\[
    Q_\lambda(x_1, \dots, x_n; q, \hbar) = 0 \quad \text{if} \quad n < l(\lambda).
\]

\subsection{Even part function}
The even part function $\mathcal{E}_\lambda(\bs x)$ is defined (see for example \cite[VI.7, Ex.~4(i)]{macdonald1995symmetric}) by
\begin{equation}\label{evenpart}
    \mathcal{E}_\lambda(\bs x) = \Phi(\bs x)^{-1} \sum_{\substack{\mu \text{ even} \\ \lambda \prec \mu}} b_\mu^{\mathrm{el}} P_{\lambda/\mu}(\bs x),
\end{equation}
where
\[
    \Phi(\bs x):=\sum_{\nu' \text{ even}} b_\nu^{\text{el}} P_\nu(\bs x).
\]
When specialized to a single variable $\bs x = (x_1)$,
\begin{equation}\label{evenpart-single}
    \mathcal{E}_\lambda(\bs x) \bigg|_{\bs x=(x_1)} = b_{e(\lambda)}^{\mathrm{el}} \psi_{e(\lambda)/\lambda} x_1^{|e(\lambda)| - |\lambda|},
\end{equation}
where $e(\lambda)$ is the unique partition such that $e(\lambda)'$ is even and $\lambda \prec e(\lambda)$.

\subsection{Inversion and branching}
The inversion symmetry states that
\begin{equation}\label{macinv}
    P_{(a)}(x_1^{-1}, \dots, x_n^{-1}; q, \hbar) = (x_1 \cdots x_n)^{-a} P_{(a^{n-1})}(x_1, \dots, x_n; q, \hbar).
\end{equation}
Because we are unable to find a reference for this formula, we include a short proof of it below.
\begin{proof}
    For the partition $(a)$, the Macdonald polynomial $P_{(a)}(x_1, \dots, x_n; q, \hbar)$ is given by
    \[
        P_{(a)}(x_1, \dots, x_n; q, \hbar) = \sum_{i=1}^n x_i^a \prod_{\substack{j=1 \\ j \neq i}}^n \frac{x_i - \hbar x_j}{x_i - x_j},
    \]
    and for the partition $(a^{n-1})$ the Macdonald polynomial is given by
    \[
        P_{(a^{n-1})}(x_1, \dots, x_n; q, \hbar) = \sum_{i=1}^n x_1^a \cdots x_{i-1}^a x_{i+1}^a \cdots x_n^a \prod_{\substack{j=1 \\ j \neq i}}^n \frac{x_j - \hbar  x_i}{x_j - x_i},
    \]
    both of which can be derived from the relationship between the symmetric and nonsymmetric Macdonald polynomials, see e.g. Section 2.1 of \cite{colmenarejo2024cfunctions}.
   
    So we have
    \[
        P_{(a)}(x_1^{-1}, \dots, x_n^{-1}; q, \hbar) = \sum_{i=1}^n x_i^{-a} \prod_{\substack{j=1 \\ j \neq i}}^n \frac{x_i^{-1} - \hbar x_j^{-1}}{x_i^{-1} - x_j^{-1}}.
    \]
    Note that
    \[
        \frac{x_i^{-1} - \hbar x_j^{-1}}{x_i^{-1} - x_j^{-1}}  = \frac{x_j - \hbar x_i}{x_j-x_i},
    \]
    so we get
    \begin{align*}
       & P_{(a)}(x_1^{-1}, \dots, x_n^{-1}; q, \hbar)  = \sum_{i=1}^n x_i^{-a} \prod_{\substack{j=1 \\ j \neq i}}^n \frac{x_j - \hbar x_i}{x_j - x_i} \\
        &= (x_1 x_2\ldots x_n)^{-a} \sum_{i=1}^n x_1^{a}\ldots x_{i-1}^{a} x_{i+1}^{a} \cdots x_n^{a} \prod_{\substack{j=1 \\ j \neq i}}^n \frac{x_j - \hbar x_i}{x_j - x_i} \\
        &= P_{(a^{n-1})}(x_1,\ldots,x_n; q,\hbar)
    \end{align*}
    as desired.
\end{proof}

The relation between skew-Macdonald $P$ and $Q$~polynomials is given by
\[
    Q_{\lambda/\mu} = b_\lambda b_\mu^{-1} P_{\lambda/\mu}.
\]
The branching property allows us to express Macdonald polynomials evaluated on a union of variables as a sum of products of skew polynomials:
\begin{align} \nonumber
    Q_\lambda(\bs x, \bs z) = \sum_\mu Q_{\lambda/\mu}(\bs x) Q_\mu(\bs z), \\
    P_\lambda(\bs x, \bs z) = \sum_\mu P_{\lambda/\mu}(\bs x) P_\mu(\bs z). \label{branching}
\end{align}
For single-variable specializations and fixed parameters, the polynomials behave as monomials scaled by appropriate coefficients:
\[
    P_{(r)}(x; q, \hbar) = x^r, \quad Q_{(r)}(x; q, \hbar) = \frac{(\hbar)_r}{(q)_r} x^r.
\]
The $q$-shifted factorial for negative indices is related to positive indices by
\begin{equation}\label{qpochinv}
    (a; q)_{-n} = \frac{(-q/a)^n q^{\binom{n}{2}}}{(q/a; q)_n}.
\end{equation}

\section{Stable type \texorpdfstring{$D$}{D} quiver representations}\label{sec:stable-quiver-rep}

In this section, we study zero-dimensional type $D$ quiver varieties. In other words, we explicitly characterize all stable type $D$ quiver representations satisfying the moment map equations when there is a single framing at a minuscule vertex. Recall the definition of the $D_n$ quiver:

\begin{definition}
    The quiver $D_n = (Q_0, Q_1)$ is defined as follows:
    \begin{itemize}
        \item $Q_0 = \{1, 2, \dots, n\}$,
        \item $i \to i+1 \in Q_1$ for all $i \in [1, n-2]$,
        \item $n-2 \to n \in Q_1$.
    \end{itemize}
\end{definition}

\begin{figure}[H]
    \centering
    \[\begin{tikzcd}
            &&& \bullet \\
            \bullet & \bullet & \bullet \\
            &&& \bullet
            \arrow[from=2-1, to=2-2]
            \arrow[from=2-2, to=2-3]
            \arrow[from=2-3, to=1-4]
            \arrow[from=2-3, to=3-4]
        \end{tikzcd}\]
    \caption{Dynkin diagram for $D_5$}
    \label{fig:d5-diagram}
\end{figure}
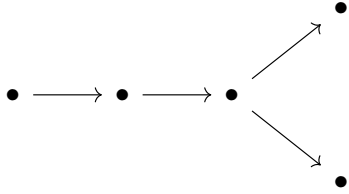

\subsection{Case of \texorpdfstring{$\msf w = (0^{n-1}, 1)$}{w = (0^{n-1}, 1)}}\label{Dposet1}

We start with the case of a single framing at a spin node, which in our notation is either vertex $n-1$ or vertex $n$. Due to the symmetry of the $D_{n}$ quiver, it does not matter which we consider. So for definiteness, assume that $\msf{w}=\delta_{n}$.

Let $\mc{M}= \mc{M}_{D_n}(\msf v, \delta_{n})$ be a quiver variety. Recall from Proposition~\ref{minframing} that $\mc{M}$ is nonempty if and only if it is a single point. By \cite{nakajima1994instantons}, such $\msf{v}$ correspond to weight spaces of the spin representation with highest weight $\omega_{n}$. We require an explicit description of the $\msf{v}$ for which $\mc{M}$ is a point, along with a special representative of the point as a quiver representation.

We recall the formula for the dimension of an arbitrary quiver variety.

\begin{proposition}[Dimension of {$\mc{M}_Q(\msf{v}, \msf{w})$}, \cite{nakajima1994instantons}]\label{prop:mq-dim}
    The dimension of the space $\mc{M}_Q(\msf v, \msf w)$ is given as follows:
    \[
        \frac{1}{2} \dim_\CC \mc M_Q(\msf v, \msf w) = \sum_{e \in Q_1} \msf v_{o(e)} \msf v_{i(e)} + \sum_{i \in Q_0} \msf v_i \msf w_i - \sum_{i \in Q_0} \msf v_i^2.
    \]
\end{proposition}
In the case of $\msf{w}=\delta_{n}$, the dimension formula is
\[
    \dim_{\CC} \mc{M} = 2 \left(\sum_{e \in Q_1} \msf v_{o(e)} \msf v_{i(e)} + \msf v_n - \sum_{i \in Q_0} \msf v_i^2\right).
\]
The possible $\msf{v}$ for which $\mc{M}$ is nonempty are explicitly described by the following lemma.

\begin{lemma}\label{lem:char-valid-v:w-last}
    The quiver variety $\mc{M}_{D_n}(\msf v, \delta_{n})$ is nonempty if and only if
    \begin{itemize}
        \item $\msf v_{i+1} = \msf v_i$ or $\msf v_{i+1} = \msf v_i + 1$ for all $i \in [1, n-3]$,
        \item $\msf v_{n-2} = \msf v_{n-1} + \msf v_n$ or $\msf v_{n-2} + 1 = \msf v_{n-1} + \msf v_n$,
        \item $\msf v_n = \msf v_{n-1}$ or $\msf v_n = \msf v_{n-1} + 1$.
    \end{itemize}
    There are exactly $2^{n-1}$ such $\msf{v}$.
\end{lemma}
\begin{proof}
    That such $\msf{v}$ give rise to a zero dimensional quiver variety follows from the dimension formula. It is straightforward to see that there are $2^{n-1}$ such $\msf{v}$ satisfying the conditions stated in the lemma. By \cite{nakajima1994instantons}, the number of $\msf{v}$ such that $\mc{M}_{D_n}(\msf v, \delta_{n})$ is nonempty is the dimension of the spin representation with highest weight $\omega_{n}$, which is known to be $2^{n-1}$. We have thus exhibited all such $\msf{v}$.
\end{proof}

\begin{figure}[H]
    \centering
    \begin{subfigure}[t]{0.3\textwidth}
        \centering
        \begin{tikzpicture}[
                dot/.style={circle, fill, inner sep=2pt}
            ]
            \node[dot] (v1) at (-2,2) {};
            \node[dot] (v2) at (-1,1) {};
            \node[dot] (v3) at (0,0) {};
            \node[dot] (v4) at (-1,3) {};
            \node[dot] (v5) at (0,2) {};

            \draw[->] (v3) -- (v2);
            \draw[->] (v2) -- (v1);
            \draw[->] (v5) -- (v2);
            \draw[->] (v1) -- (v4);
            \draw[->] (v5) -- (v4);

            \node[above=0.1cm of v1] {$d_{1,1}$};
            \node[above=0.1cm of v2] {$d_{2,1}$};
            \node[above=0.1cm of v3] {$d_{4,1}$};
            \node[above=0.1cm of v4] {$d_{2,2}$};
            \node[above=0.1cm of v5] {$d_{3,1}$};
        \end{tikzpicture}
    \end{subfigure}%
    \begin{subfigure}[t]{0.3\textwidth}
        \centering
        \begin{tikzpicture}[
                dot/.style={circle, fill, inner sep=2pt}
            ]
            \node[dot] (v1) at (-2,2) {};
            \node[dot] (v2) at (-1,1) {};
            \node[dot] (v3) at (0,0) {};
            \node[dot] (v4) at (-1,3) {};
            \node[dot] (v5) at (0,2) {};

            \draw[->] (v3) -- (v2);
            \draw[->] (v2) -- (v1);
            \draw[->] (v5) -- (v2);
            \draw[->] (v1) -- (v4);
            \draw[->] (v5) -- (v4);

            \node[above=0.1cm of v1] {$z_1$};
            \node[above=0.1cm of v2] {$z_2$};
            \node[above=0.1cm of v3] {$z_4$};
            \node[above=0.1cm of v4] {$z_2$};
            \node[above=0.1cm of v5] {$z_3$};
        \end{tikzpicture}
    \end{subfigure}%
    \begin{subfigure}[t]{0.3\textwidth}
        \centering
        \begin{tikzpicture}[
                dot/.style={circle, fill, inner sep=2pt}
            ]
            \node[dot] (v1) at (-2,2) {};
            \node[dot] (v2) at (-1,1) {};
            \node[dot] (v3) at (0,0) {};
            \node[dot] (v4) at (-1,3) {};
            \node[dot] (v5) at (0,2) {};

            \draw[->] (v3) -- (v2);
            \draw[->] (v2) -- (v1);
            \draw[->] (v5) -- (v2);
            \draw[->] (v1) -- (v4);
            \draw[->] (v5) -- (v4);

            \node[above=0.1cm of v1] {$\hbar^2$};
            \node[above=0.1cm of v2] {$\hbar$};
            \node[above=0.1cm of v3] {$1$};
            \node[above=0.1cm of v4] {$\hbar^2$};
            \node[above=0.1cm of v5] {$\hbar$};
        \end{tikzpicture}
    \end{subfigure}
    \caption{A $D_4$ poset of type $\omega_4$, with $\msf v = (1, 2, 1, 1)$}
    \label{fig:d4-poset-omega4}
\end{figure}
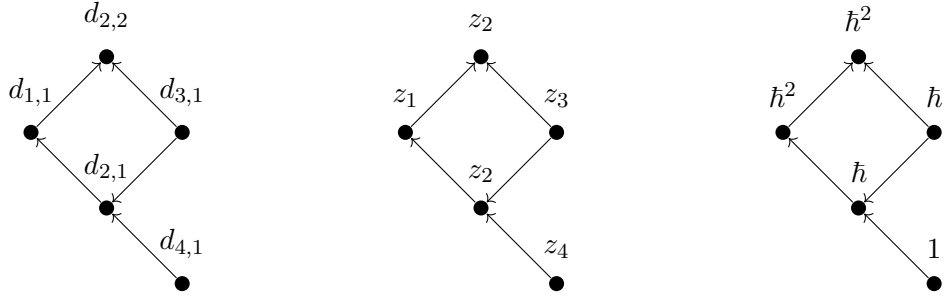

We now exhibit explicit quiver representations which are representatives of the unique point in $\mc{M}_{D_{n}}(\msf{v},\delta_{n})$ for the $\msf{v}$ described above. We will build such quiver representations from posets.

\begin{definition}\label{dspinposet}
    The type $D_n$ spin poset $\mathcal{P} = \left(P, \leq\right)$ is defined by
    \[
        P = \{d_{i,j}\}_{(i,j) \in I},
    \]
    where $I = \{(i,j) : i \in [1, n-2], j \in [1, i]\}\,\cup\,\{(n-1,j) : j \in [1, \floor{\frac{n+1}{2}}]\}\,\cup\,\{(n, j) : j \in [1, \floor{\frac{n+2}{2}}]\}$.
    The partial order $\leq$ is defined by the following relations:
    \[
        d_{i,j} \geq d_{i+1,j} \quad \text{for all } i \in [1, n-3], \ j \in [1, i],
    \]
    \[
        d_{i,j} \leq d_{i+1,j+1} \quad \text{for all } i \in [1, n-3], \ j \in [1, i],
    \]
    \[
        d_{n-2,2j} \geq d_{n-1,j}, \quad  d_{n-2,2j-1} \leq d_{n-1,j} \quad \text{for all } j \in \left[1, \frac{n+1}{2}\right],
    \]
    \[
        d_{n-2,2j-1} \geq d_{n,j} \quad \text{for all } j \in \left[1, \frac{n+2}{2}\right],
    \]
    \[
        d_{n-2,2j-2} \leq d_{n,j} \quad \text{for all } j \in \left[2, \frac{n+2}{2}\right].
    \]
    Moreover, the coloring function satisfies $\col(d_{i,j}) = i$ for all $i \in [1, n]$ and for all $j$, that is, each element $d_{i,j}$ is colored by vertex $i$. The weight function satisfies
    \[
        \weight(d_{i,j}) = \hbar^{n-i+j-2} \quad \text{for all } i \in [1, n-2], j \in [1, i],
    \]
    \[
        \weight(d_{n-1,j}) = \hbar^{2j-1} \quad \text{for all } j \in \left[1, \frac{n+1}{2}\right],
    \]
    and
    \[
        \weight(d_{n,j}) = \hbar^{2j-2} \quad \text{for all } j \in \left[1, \frac{n+2}{2}\right].
    \]
\end{definition}

\begin{definition}
    An order ideal of the type $D_{n}$ spin poset, along with the induced coloring and weight function, is called a $D_{n}$ poset of type $\omega_{n}$.
\end{definition}

Given a $D_n$ spin poset $\mc{P} = (P, \le)$ of type $\omega_{n}$, we construct the corresponding quiver representation as follows.
For each vertex $i \in Q_0$, assign a vector space $V_i$ over $\CC$ whose dimension is equal to the number of elements in $P$ colored by $i$, i.e., $\dim V_i = \left|\col^{-1}(i)\right|$.
For each element $d_{i,j} \in P$, assign a basis vector $e_{i,j} \in V_i$. We define linear maps $X_{e}:V_{o(e)} \to V_{i(e)}$ and $Y_{e}: V_{i(e)} \to V_{o(e)}$ for each edge $e \in Q_{1}$ in terms of the vectors $e_{i,j}$ by asserting that
\begin{itemize}
    \item $X_{e}$ maps $e_{o(e),j}$ to $e_{i(e),k}$ whenever there is a covering relation $d_{o(e),j}\leq d_{i(e),k}$,
    \item $Y_{e}$ maps $e_{i(e),j}$ to $e_{o(e),k}$ whenever there is a covering relation $d_{i(e),j}\leq d_{o(e),k}$,
    \item the remaining matrix coefficients are zero.
\end{itemize}
We also define $A_{n} : \CC \to V_{n}$ by $A_{n}(1)=d_{n,1}$ and $A_{i}=0$ for $i < n$. We set $B_{i}=0$ for all $i$. Overall, we obtain a collection of linear maps $(X,Y,A,B)$ which is the corresponding quiver representation of the $D_n$ poset $\mathcal P$ of type $\omega_n$.

\begin{proposition}\label{posettorep}
    The quiver representation $(X,Y,A,B)$ associated to a $D_{n}$ poset of type $\omega_{n}$ is stable and satisfies the moment map equations.
\end{proposition}
\begin{proof}
    Recall Proposition~\ref{stability} which gives a characterization of stability. Combinatorially, this means that one is able to reach any basis element from the framing lowest basis element by traversing through arrows. This follows from the fact that each nonempty $D_{n}$ poset of type $\omega_{n}$ has a minimal element, namely $d_{n,1}$. It remains to show that the quiver representation associated to $D_n$ poset of type $\omega_n$ satisfies $\mu = 0$ at each vertex. We need the right-hand side of Equation~\ref{eq:moment-map} to vanish identically:
    \[
        \sum_{\substack{e \in Q_1 \\ i(e)=i}} X_e Y_e-\sum_{\substack{e \in Q_1 \\ o(e)=i}} Y_e X_e+A_i B_i = 0 \quad \text{for all } i \in Q_0.
    \]
    For the $D_{n}$ spin poset, this follows directly from the definitions, and is inherited by its order ideals.
\end{proof}

Putting together the results of this section, we obtain the following.

\begin{theorem}\label{posetthm1}
    The nonempty (and hence zero-dimensional) quiver varieties $\mc{M}_{D_{n}}(\msf{v},\delta_{n})$ are in canonical bijection with $D_{n}$ posets of type $\omega_{n}$.
\end{theorem}

\begin{remark}
    Considering framing $\delta_{n-1}$ gives rise to a theorem analogous to the previous one, giving a canonical bijection between the nonempty quiver varieties $\mc{M}_{D_{n}}(\msf{v},\delta_{n-1})$ and $D_{n}$ posets of type $\omega_{n-1}$. The latter are identical as posets to $D_{n}$ posets of type $\omega_{n}$. The weight function is identical and the coloring function differs by swapping colors $n-1$ and $n$.
\end{remark}

\subsection{Case of \texorpdfstring{$\msf w = (1, 0^{n-1})$}{w = (1, 0^{n-1})}}\label{Dposet2}

Now we consider the case where $\msf{w}=\delta_{1}$, which we call the fundamental node. As in the previous subsection, Proposition~\ref{minframing} implies that $\mc{M}_{D_{n}}(\msf{v},\delta_{1})$ is nonempty if and only if it is a single point. Such dimension vectors are described by the following lemma.

\begin{lemma}\label{lem:framing-first-configs}
    The quiver variety $\mc{M}_{D_{n}}(\msf{v},\delta_{1})$ is nonempty if and only if $\msf{v}$ is one of the following vectors:
    \begin{itemize}
        \item $\msf{v}=(1^{k},0^{n-k})$ for $k \in [0,n]$,
        \item $\msf{v}=(1^{k},2^{n-k-2},1,1)$ for $k \in [0,n-3]$,
        \item $\msf{v}=(1^{n-2},0,1)$.
    \end{itemize}
    There are exactly $2n$ such $\msf{v}$.
\end{lemma}

\begin{proof}
    From Proposition~\ref{prop:mq-dim}, one can see that the dimension vectors in the statement of the lemma give rise to a zero dimensional quiver variety. There are a total of $(n+1)+(n-2)+1=2n$ such vectors. Since this is the dimension of the first fundamental representation of type $D_{n}$, the results of \cite{nakajima1994instantons} imply that we have exhibited all such $\msf{v}$.
\end{proof}

\begin{figure}[H]
    \centering
    \begin{subfigure}[t]{0.33\textwidth}
        \centering
        \begin{tikzpicture}[
                dot/.style={circle, fill, inner sep=2pt}
            ]
            \node[dot] (v1) at (0,0) {};
            \node[dot] (v2) at (1,1) {};
            \node[dot] (v3) at (2,2) {};
            \node[dot] (v4) at (3,3) {};
            \node[dot] (v5) at (1,3) {};
            \node[dot] (v6) at (2,4) {};
            \node[dot] (v7) at (1,5) {};
            \node[dot] (v8) at (0,6) {};

            \draw[->] (v1) -- (v2);
            \draw[->] (v2) -- (v3);
            \draw[->] (v3) -- (v4);
            \draw[->] (v3) -- (v5);
            \draw[->] (v4) -- (v6);
            \draw[->] (v5) -- (v6);
            \draw[->] (v6) -- (v7);
            \draw[->] (v7) -- (v8);

            \node[above=0.1cm of v1] {$d_{1,1}$};
            \node[above=0.1cm of v2] {$d_{2,1}$};
            \node[above=0.1cm of v3] {$d_{3,1}$};
            \node[above=0.1cm of v4] {$d_{4,1}$};
            \node[above=0.1cm of v5] {$d_{5,1}$};
            \node[above=0.1cm of v6] {$d_{3,2}$};
            \node[above=0.1cm of v7] {$d_{2,2}$};
            \node[above=0.1cm of v8] {$d_{1,2}$};
        \end{tikzpicture}
    \end{subfigure}%
    \begin{subfigure}[t]{0.33\textwidth}
        \centering
        \begin{tikzpicture}[
                dot/.style={circle, fill, inner sep=2pt}
            ]
            \node[dot] (v1) at (0,0) {};
            \node[dot] (v2) at (1,1) {};
            \node[dot] (v3) at (2,2) {};
            \node[dot] (v4) at (3,3) {};
            \node[dot] (v5) at (1,3) {};
            \node[dot] (v6) at (2,4) {};
            \node[dot] (v7) at (1,5) {};
            \node[dot] (v8) at (0,6) {};

            \draw[->] (v1) -- (v2);
            \draw[->] (v2) -- (v3);
            \draw[->] (v3) -- (v4);
            \draw[->] (v3) -- (v5);
            \draw[->] (v4) -- (v6);
            \draw[->] (v5) -- (v6);
            \draw[->] (v6) -- (v7);
            \draw[->] (v7) -- (v8);

            \node[above=0.1cm of v1] {$z_1$};
            \node[above=0.1cm of v2] {$z_2$};
            \node[above=0.1cm of v3] {$z_3$};
            \node[above=0.1cm of v4] {$z_4$};
            \node[above=0.1cm of v5] {$z_5$};
            \node[above=0.1cm of v6] {$z_3$};
            \node[above=0.1cm of v7] {$z_2$};
            \node[above=0.1cm of v8] {$z_1$};
        \end{tikzpicture}
    \end{subfigure}%
    \begin{subfigure}[t]{0.33\textwidth}
        \centering
        \begin{tikzpicture}[
                dot/.style={circle, fill, inner sep=2pt}
            ]
            \node[dot] (v1) at (0,0) {};
            \node[dot] (v2) at (1,1) {};
            \node[dot] (v3) at (2,2) {};
            \node[dot] (v4) at (3,3) {};
            \node[dot] (v5) at (1,3) {};
            \node[dot] (v6) at (2,4) {};
            \node[dot] (v7) at (1,5) {};
            \node[dot] (v8) at (0,6) {};

            \draw[->] (v1) -- (v2);
            \draw[->] (v2) -- (v3);
            \draw[->] (v3) -- (v4);
            \draw[->] (v3) -- (v5);
            \draw[->] (v4) -- (v6);
            \draw[->] (v5) -- (v6);
            \draw[->] (v6) -- (v7);
            \draw[->] (v7) -- (v8);

            \node[above=0.1cm of v1] {$1$};
            \node[above=0.1cm of v2] {$1$};
            \node[above=0.1cm of v3] {$1$};
            \node[above=0.1cm of v4] {$1$};
            \node[above=0.1cm of v5] {$1$};
            \node[above=0.1cm of v6] {$\hbar$};
            \node[above=0.1cm of v7] {$\hbar^2$};
            \node[above=0.1cm of v8] {$\hbar^3$};
        \end{tikzpicture}
    \end{subfigure}
    \caption{A $D_5$ poset of type $\omega_1$, with $\msf v = (2, 2, 2, 1, 1)$}
    \label{fig:d4-poset-omega1}
\end{figure}
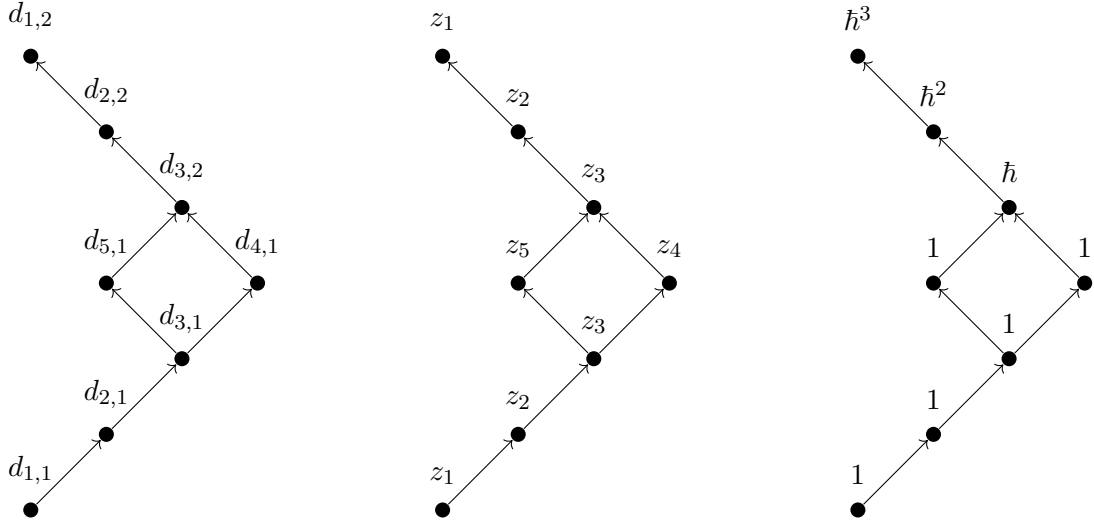

As above, we next exhibit quiver representations which are representatives of the unique point in $\mc{M}_{D_{n}}(\msf{v},\delta_{1})$ for the $\msf{v}$ described above.

\begin{definition}\label{dfundposet}
    The type $D_n$ fundamental representation poset $\mathcal{P} = \left(P, \leq\right)$ is defined by
    \[
        P = \{d_{i,j}\}_{(i,j) \in I},
    \]
    where $I = \{(i,j) : i \in [1, n-2], j \in [1, 2]\} \cup \{(n-1,1), (n,1)\}$.
    The partial order $\leq$ is defined by the following relations:
    \[
        d_{i,1} \leq d_{i+1,1} \quad \text{for all } i \in [1, n-2],
    \]
    \[
        d_{i,2} \geq d_{i+1,2} \quad \text{for all } i \in [1, n-3],
    \]
    \[
        d_{n-2,1} \leq d_{n,1}, \quad d_{n-2,2} \geq d_{n,1}, \quad d_{n-2,2} \geq d_{n-1,1}.
    \]
    Moreover, the coloring function satisfies $\col(d_{i,j}) = i$ for all $i \in [1, n]$ and for all $j$, that is, each element $d_{i,j}$ is colored by vertex $i$. The weight function satisfies $\weight(d_{i,1}) = 1$ for all $i \in [1, n]$ and $\weight(d_{i,2}) = \hbar^{n-i-1}$ for all $i \in [1, n-2]$.
\end{definition}

\begin{definition}
    An order ideal of the type $D_{n}$ fundamental poset, along with the induced coloring and weight function, is called a $D_{n}$ poset of type $\omega_{1}$.
\end{definition}

Given a $D_n$ poset $\mc{P} = (P, \le)$ of type $\omega_{1}$, we construct the corresponding quiver representation in the same way as before.
For each vertex $i \in Q_0$, assign a vector space $V_i$ over $\CC$ whose dimension is equal to the number of elements in $P$ colored by $i$, i.e., $\dim V_i = \left|\col^{-1}(i)\right|$. For each element $d_{i,j} \in P$, assign a basis vector $e_{i,j} \in V_i$. We define linear maps $X_{e}:V_{o(e)} \to V_{i(e)}$ and $Y_{e}: V_{i(e)} \to V_{o(e)}$ for each edge $e \in Q_{1}$ in terms of the vectors $e_{i,j}$ by asserting that
\begin{itemize}
    \item $X_{e}$ maps $e_{o(e),j}$ to $e_{i(e),k}$ whenever there is a covering relation $d_{o(e),j}\leq d_{i(e),k}$,
    \item $Y_{e}$ maps $e_{i(e),j}$ to $e_{o(e),k}$ whenever there is a covering relation $d_{i(e),j}\leq d_{o(e),k}$,
    \item the remaining matrix coefficients are zero.
\end{itemize}

We also define $A_{1} : \CC \to V_{1}$ by $A_{1}(1)=d_{1,1}$ and $A_{i}=0$ for $i > 1$. We set $B_{i}=0$ for all $i$. Overall, we obtain a collection of linear maps $(X,Y,A,B)$ which is the corresponding quiver representation of the $D_n$ poset $\mathcal P$ of type $\omega_1$.

\begin{proposition}
    The quiver representation associated to a $D_{n}$ poset of type $\omega_{1}$ is stable and satisfies the moment map equations.
\end{proposition}
\begin{proof}
    The proof is analogous to the proof of Proposition~\ref{posettorep}.
\end{proof}

Putting together the results of this section, we obtain the following.

\begin{theorem}\label{posetthm2}
    The nonempty (and hence zero-dimensional) quiver varieties $\mc{M}_{D_{n}}(\msf{v},\delta_{1})$ are in canonical bijection with $D_{n}$ posets of type $\omega_{1}$.
\end{theorem}

\section{Vertex functions}\label{sec:type-d-vertex-functions}

In this section, we will review the definition of vertex functions, our main objects of interest. Vertex functions for Nakajima quiver varieties were defined by A.~Okounkov in \cite{okounkov2017lectures} using quasimap machinery developed in \cite{cfkm2014quasimapsgit}. Later we will restrict to the case of type $D$ quiver varieties and spell out the definitions more explicitly in this context.

\subsection{General definitions}\label{subsec:qmdef}

Let $\mc{M}:=\mc{M}_{Q}(\msf v, \msf w)$ be a quiver variety. Recall that $\mc{M}$ is defined as $\mu^{-1}(0)/\!\!/_{\theta} G_{\msf v}$. Then there is an open embedding
\[
    \mc{M} \hookrightarrow [\mu^{-1}(0)/G_{\msf v}]=:\mf{M}.
\]
A quasimap from $\mathbb{P}^{1}$ to $\mc{M}$ is by definition a map $f$ from $\mathbb{P}^{1} \to \mf{M}$. A quasimap is said to be stable if $f(p) \in \mc{M}$ for all but finitely many $p \in \mathbb{P}^{1}$. Points $p$ such that $f(p) \in \mf{M} \setminus \mc{M}$ are called singularities of $f$. Let $\msf{QM}$ be the stack parameterizing stable quasimaps from $\mathbb{P}^{1}$ to $\mc{M}$.

Unpacking the definition, one sees that the data of a quasimap $f$ consists vectors bundles $\mc{V}_{i}$ and topologically trivial bundles $\mc{W}_{i}$ for $i \in Q_0$ over $\mathbb{P}^{1}$ along with a section of a certain bundle built from $\mc{V}_{i}$ and $\mc{W}_{i}$; see Section~4 of \cite{okounkov2017lectures}. The degree of the quasimap $f$ is defined to be
\[
    \deg f:=(\deg \mc{V}_{i})_{i \in Q_0} \in \mathbb{Z}^{Q_0}.
\]
For $\bs d \in \mathbb{Z}^{Q_0}$, let $\msf{QM}^{\bs d}$ be the stack parameterizing degree $\bs d$ stable quasimaps. Thus
\[
    \msf{QM}= \bigsqcup_{\bs d \in \mathbb{Z}^{Q_0}} \msf{QM}^{\bs d}.
\]
Let $\msf{QM}_{\text{ns}, \infty}^{\bs d} \subset \msf{QM}^{\bs d}$ be the space of quasimaps which are nonsingular at $\infty \in \PP^{1}$. There is a diagram
\[\begin{tikzcd}
        & \msf{QM}_{\text{ns}, \infty}^{\bs d} \arrow{dl}{\ev_{0}} \arrow{dr}{\ev_{\infty}} & \\
        \mf{M} & & \mc{M}
    \end{tikzcd}\]
where the maps are given by evaluating quasimaps at $0$ and $\infty$. The action of the maximal torus $\msf T \subset \Aut(\mc{M})$ induces an action on $\msf{QM}^{\bs d}_{\text{ns},\infty}$. In addition, there is an action of $\CC^{\times}$ on $\msf{QM}^{\bs d}_{\text{ns},\infty}$ induced by the usual action on $\PP^{1}$. We denote this torus by $\CC_{q}^{\times}$. The restriction of $\ev_{\infty}$ to $\left(\msf{QM}^{\bs d}_{\text{ns}, \infty}\right)^{\mathbb{C}^{\times}_{q}}$ is known to be proper; see Section~7.3 of \cite{okounkov2017lectures}. Hence the pushforward $\ev_{\infty,*}$ can be defined in localized equivariant $K$-theory.

It is known that $\msf{QM}^{\bs d}$ is equipped with a canonical perfect obstruction theory, which gives rise to a virtual structure sheaf $\mc{O}_{\text{vir}}^{\bs d}$. For technical reasons, it is better to study the symmetrized virtual structure sheaf $\hat{\mc{O}}_{\text{vir}}^{\bs{d}}$, which differs from the virtual structure sheaf by a twist by certain line bundle; see Section~6 of \cite{okounkov2017lectures}. This twist depends on a choice of polarization of $\mc{M}$, and we make the choice \eqref{polarization}.

\begin{definition}[\cite{okounkov2017lectures}]
    Let $\tau \in K_{\msf{T} \times \CC^{\times}_{q}}(\mf{M})$. The vertex function with descendant $\tau$ is the formal power series
    \[
        V^{(\tau)}(\bs{z}) = \sum_{\bs{d}} \ev_{\infty, *}\left(\hat{\mc{O}}_{\text{vir}}^{\bs{d}} \otimes \ev_{0}^{*}(\tau)\right) \bs{z}^{\bs{d}} \in K_{\msf T \times \CC_q^{\times}}(\mc{M})_{loc}[[z]],
    \]
    where $\bs{z}^{\bs d}:=\prod_{i \in Q_0} z_{i}^{d_{i}}$.
\end{definition}

The formal parameters $z_{i}$ are called K\"ahler parameters.
To write explicit formulas for vertex functions, we will need some additional notation. Let $\varphi(x) := \prod_{i \geq 0}\left(1-x q^i\right)$. The $q$-Pochhammer symbol is defined by
\[
    (x)_d := \frac{\varphi(x)}{\varphi(x q^d)}= \begin{cases}
        1                                 & \quad\text{if }d = 0, \\
        \prod_{i=0}^{d-1}(1-x q^{i})      & \quad\text{if }d > 0, \\
        \prod_{i=1}^{-d}(1-x q^{-i})^{-1} & \quad\text{if }d <0,
    \end{cases}
\]
and the $q$-binomial series is defined by
\[
    F(z) := \sum_{d=0}^{\infty} \frac{(\hbar)_d}{(q)_d} z^d.
\]

\begin{proposition}[\cite{gasperrahman2004}]
    The $q$-binomial series $F(z)$ satisfies the following identity:
    \[
        F(z) = \prod_{n=0}^{\infty} \frac{1-z \hbar q^n}{1-z q^n}.
    \]
\end{proposition}

\subsection{Type \texorpdfstring{$D$}{D} vertex functions}

Let $Q$ be the type $D_{n}$ quiver and let $\msf{v},\msf{w} \in \NN^{Q_0}$. We are interested in explicit formulas for vertex functions, so we will restrict our attention to cases where $\mc{M}^{\msf{T}}$ consists of isolated points. By Proposition~\ref{minframing}, it is sufficient to assume that $\msf{w}_{i}=0$ unless $i$ is a minuscule vertex. We make this assumption now. In other words, we only allow nontrivial framings at vertices $1$, $n-1$, and $n$.

We have $\msf{T}=(\CC^{\times})^{|\msf{w}|} \times \CC^{\times}_{\hbar}$. The set $\mc{M}^{\msf{T}}$ is described completely and explicitly by combining \eqref{tens} and the results of Section~\ref{sec:stable-quiver-rep}.

Let $p \in \mc{M}^{\msf{T}}$. In this description, the fixed point $p$ is given by a $|\msf{w}|$-tuple
\[
    \left(\pst_{p}^{i,j}\right)_{\substack{i \in Q_{0} \\ 1 \leq j \leq \msf{w}_{i}}}
\]
of colored posets where $\pst_{p}^{i,j}$ is of type $\omega_{i}$ for all $j \in \{1,\ldots,\msf{w}_{i}\}$ such that the total number of elements of color $k$ is equal to $\msf{v}_{k}$ for all $k \in Q_0$. We denote
\[
    \pst_{p} = \bigsqcup_{\substack{i \in Q_{0} \\ 1 \leq j \leq \msf{w}_{i}}} \pst_{p}^{i,j}.
\]
There is a coloring function $c_{p}:\pst_{p} \to Q_{0}$ which restricts to the coloring functions of Sections~\ref{Dposet1} and~\ref{Dposet2} on each $\pst_{p}^{i,j}$. As in Sections~\ref{Dposet1} and~\ref{Dposet2}, each $\pst_{p}^{i,j}$ has a weight function $\weight_{i,j}: \pst_{p}^{i,j} \to \mathbb{Z}[\hbar^{\pm 1}]$. We define a weight function
\[
    \weight:\pst_{p} \to \text{Rep}(\msf{T})=\mathbb{Z}\left[\left\{a_{i,j}^{\pm 1}\right\}_{\substack{i \in Q_0 \\ 1 \leq j \leq \msf{w}_{i}}},\hbar^{\pm 1}\right]
\]
on $\pst_{p}$ by $\weight(x)=\weight_{i,j}(x) a_{i,j}$ for $x \in \pst_{p}^{i,j}$.

The results of Section~\ref{sec:stable-quiver-rep} imply that the $\msf{T}$-weights of the restrictions $\tb_{i}|_{p}$ are given by
\begin{equation}\label{tbweights}
    \tb_{i}|_{p}=\sum_{x \in \col_{p}^{-1}(i)} \weight(x).
\end{equation}

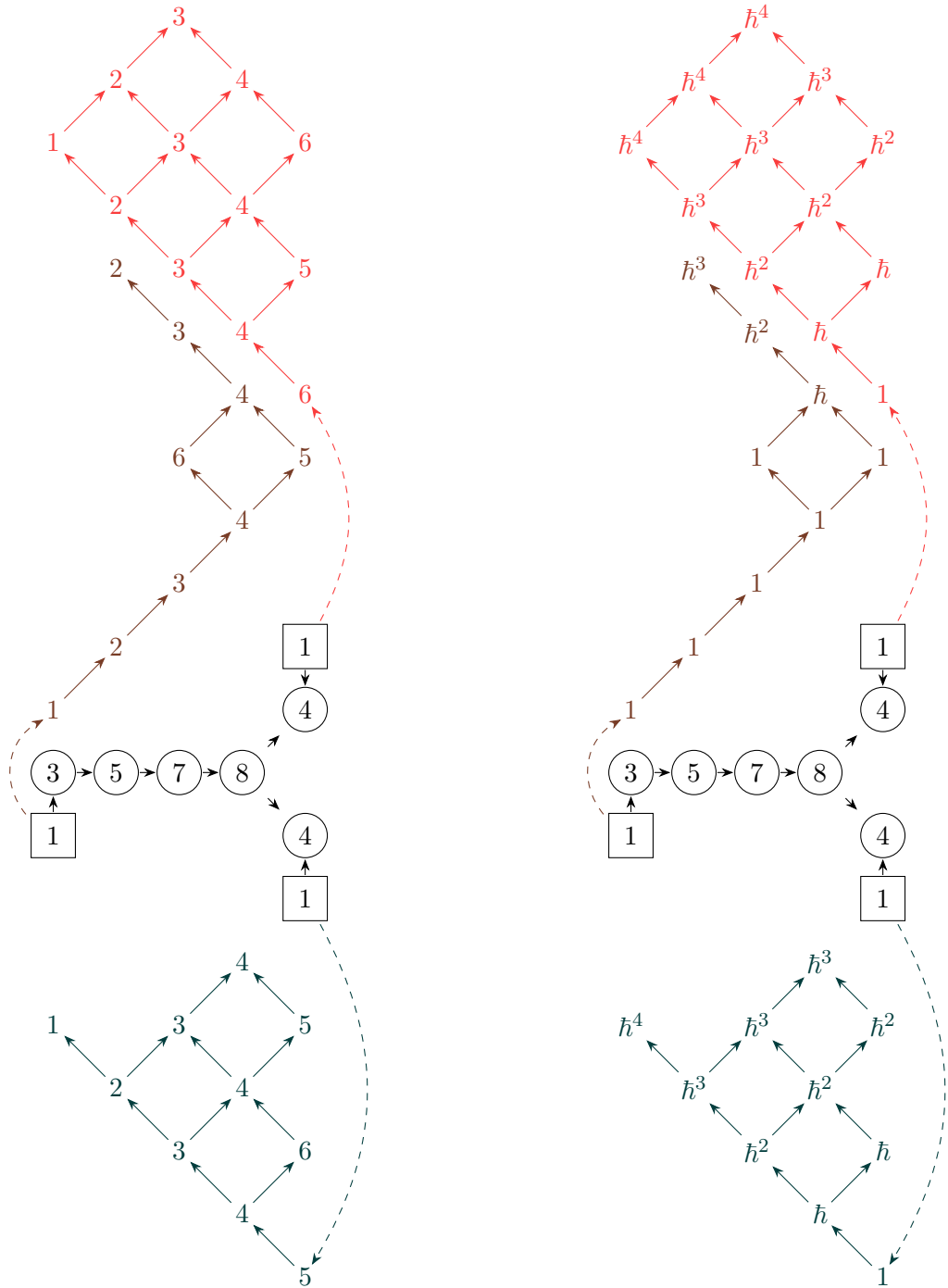
\begin{figure}[H]
    \centering
    \begin{subfigure}[b]{0.5\textwidth}
        \begin{tikzpicture}[
            x=0.9cm, y=0.9cm,
            >={Stealth},
            ->,
            shorten >=-3pt, shorten <=-3pt
            ]
            \definecolor{brownColor}{RGB}{120,60,38}
            \definecolor{pinkColor}{RGB}{250,60,60}
            \definecolor{cyanColor}{RGB}{0,60,60}

            \begin{scope}[shift={(-1,1)}, color=brownColor]
                \node (z1_l) at (1,0) {$1$};
                \node (z2) at (2,1) {$2$};
                \node (z3) at (3,2) {$3$};
                \node (z4) at (4,3) {$4$};
                \node (z6) at (3,4) {$6$};
                \node (z5) at (5,4) {$5$};
                \node (z4t) at (4,5) {$4$};
                \node (z3t) at (3,6) {$3$};
                \node (z2t) at (2,7) {$2$};

                \draw[->] (z1_l) -- (z2);
                \draw[->] (z2) -- (z3);
                \draw[->] (z3) -- (z4);
                \draw[->] (z4) -- (z6);
                \draw[->] (z4) -- (z5);
                \draw[->] (z6) -- (z4t);
                \draw[->] (z5) -- (z4t);
                \draw[->] (z4t) -- (z3t);
                \draw[->] (z3t) -- (z2t);
            \end{scope}

            \begin{scope}[shift={(0,-4)}, color=cyanColor]
                \node (z4) at (3,1) {$4$};
                \node (z1) at (0,0) {$1$};
                \node (z3) at (2,0) {$3$};
                \node (z5) at (4,0) {$5$};
                \node (z2) at (1,-1) {$2$};
                \node (z4b) at (3,-1) {$4$};
                \node (z3b) at (2,-2) {$3$};
                \node (z6) at (4,-2) {$6$};
                \node (z4c) at (3,-3) {$4$};
                \node (z5b) at (4,-4) {$5$};

                \draw[->] (z3) -- (z4);
                \draw[->] (z5) -- (z4);
                \draw[->] (z2) -- (z1);
                \draw[->] (z2) -- (z3);
                \draw[->] (z4b) -- (z3);
                \draw[->] (z4b) -- (z5);
                \draw[->] (z3b) -- (z2);
                \draw[->] (z3b) -- (z4b);
                \draw[->] (z6) -- (z4b);
                \draw[->] (z4c) -- (z3b);
                \draw[->] (z4c) -- (z6);
                \draw[->] (z5b) -- (z4c);
            \end{scope}

            \begin{scope}
                \node (n1_bot) at (0,-1) {\squared{$1$}};
                \node (n3) at (0,0) {\circled{$3$}};
                \node (n5) at (1,0) {\circled{$5$}};
                \node (n7) at (2,0) {\circled{$7$}};
                \node (n8) at (3,0) {\circled{$8$}};
                \node (n4_top) at (4,1) {\circled{$4$}};
                \node (n1_tr) at (4,2) {\squared{$1$}};
                \node (n1_br) at (4,-2) {\squared{$1$}};
                \node (n4_bot) at (4,-1) {\circled{$4$}};

                \draw[->] (n3) -- (n1_bot);
                \draw[->] (n5) -- (n3);
                \draw[->] (n7) -- (n5);
                \draw[->] (n8) -- (n7);
                \draw[->] (n4_top) -- (n8);
                \draw[->] (n4_top) -- (n1_tr);
                \draw[->] (n4_bot) -- (n1_br);
                \draw[->] (n4_bot) -- (n8);
            \end{scope}

            \begin{scope}[shift={(0,10)}, color=pinkColor]
                \node (z3_top) at (2,2) {$3$};
                \node (z2_left) at (1,1) {$2$};
                \node (z4_right) at (3,1) {$4$};
                \node (z1) at (0,0) {$1$};
                \node (z3_mid) at (2,0) {$3$};
                \node (z6) at (4,0) {$6$};
                \node (z2_left2) at (1,-1) {$2$};
                \node (z4_right2) at (3,-1) {$4$};
                \node (z3_bot) at (2,-2) {$3$};
                \node (z5) at (4,-2) {$5$};
                \node (z4_bot) at (3,-3) {$4$};
                \node (z6_bot) at (4,-4) {$6$};

                \draw[->] (z2_left) -- (z3_top);
                \draw[->] (z4_right) -- (z3_top);
                \draw[->] (z1) -- (z2_left);
                \draw[->] (z3_mid) -- (z2_left);
                \draw[->] (z3_mid) -- (z4_right);
                \draw[->] (z6) -- (z4_right);
                \draw[->] (z2_left2) -- (z1);
                \draw[->] (z2_left2) -- (z3_mid);
                \draw[->] (z4_right2) -- (z3_mid);
                \draw[->] (z4_right2) -- (z6);
                \draw[->] (z3_bot) -- (z2_left2);
                \draw[->] (z3_bot) -- (z4_right2);
                \draw[->] (z5) -- (z4_right2);
                \draw[->] (z4_bot) -- (z3_bot);
                \draw[->] (z4_bot) -- (z5);
                \draw[->] (z6_bot) -- (z4_bot);
            \end{scope}

            \draw[dashed, bend right=30, color=pinkColor] (n1_tr) to (z6_bot);
            \draw[dashed, bend left=30, color=cyanColor] (n1_br) to (z5b);
            \draw[dashed, bend left=50, color=brownColor] (n1_bot) to (z1_l);
        \end{tikzpicture}
    \end{subfigure}%
    \begin{subfigure}[b]{0.5\textwidth}
        \begin{tikzpicture}[
            x=0.9cm, y=0.9cm,
            >={Stealth},
            ->,
            shorten >=-3pt, shorten <=-3pt
            ]
            \definecolor{brownColor}{RGB}{120,60,38}
            \definecolor{pinkColor}{RGB}{250,60,60}
            \definecolor{cyanColor}{RGB}{0,60,60}

            \begin{scope}[shift={(-1,1)}, color=brownColor]
                \node (z1_l) at (1,0) {$1$};
                \node (z2) at (2,1) {$1$};
                \node (z3) at (3,2) {$1$};
                \node (z4) at (4,3) {$1$};
                \node (z6) at (3,4) {$1$};
                \node (z5) at (5,4) {$1$};
                \node (z4t) at (4,5) {$\hbar$};
                \node (z3t) at (3,6) {$\hbar^2$};
                \node (z2t) at (2,7) {$\hbar^3$};

                \draw[->] (z1_l) -- (z2);
                \draw[->] (z2) -- (z3);
                \draw[->] (z3) -- (z4);
                \draw[->] (z4) -- (z6);
                \draw[->] (z4) -- (z5);
                \draw[->] (z6) -- (z4t);
                \draw[->] (z5) -- (z4t);
                \draw[->] (z4t) -- (z3t);
                \draw[->] (z3t) -- (z2t);
            \end{scope}

            \begin{scope}[shift={(0,-4)}, color=cyanColor]
                \node (z4) at (3,1) {$\hbar^3$};
                \node (z1) at (0,0) {$\hbar^4$};
                \node (z3) at (2,0) {$\hbar^3$};
                \node (z5) at (4,0) {$\hbar^2$};
                \node (z2) at (1,-1) {$\hbar^3$};
                \node (z4b) at (3,-1) {$\hbar^2$};
                \node (z3b) at (2,-2) {$\hbar^2$};
                \node (z6) at (4,-2) {$\hbar$};
                \node (z4c) at (3,-3) {$\hbar$};
                \node (z5b) at (4,-4) {$1$};

                \draw[->] (z3) -- (z4);
                \draw[->] (z5) -- (z4);
                \draw[->] (z2) -- (z1);
                \draw[->] (z2) -- (z3);
                \draw[->] (z4b) -- (z3);
                \draw[->] (z4b) -- (z5);
                \draw[->] (z3b) -- (z2);
                \draw[->] (z3b) -- (z4b);
                \draw[->] (z6) -- (z4b);
                \draw[->] (z4c) -- (z3b);
                \draw[->] (z4c) -- (z6);
                \draw[->] (z5b) -- (z4c);
            \end{scope}

            \begin{scope}
                \node (n1_bot) at (0,-1) {\squared{$1$}};
                \node (n3) at (0,0) {\circled{$3$}};
                \node (n5) at (1,0) {\circled{$5$}};
                \node (n7) at (2,0) {\circled{$7$}};
                \node (n8) at (3,0) {\circled{$8$}};
                \node (n4_top) at (4,1) {\circled{$4$}};
                \node (n1_tr) at (4,2) {\squared{$1$}};
                \node (n1_br) at (4,-2) {\squared{$1$}};
                \node (n4_bot) at (4,-1) {\circled{$4$}};

                \draw[->] (n3) -- (n1_bot);
                \draw[->] (n5) -- (n3);
                \draw[->] (n7) -- (n5);
                \draw[->] (n8) -- (n7);
                \draw[->] (n4_top) -- (n8);
                \draw[->] (n4_top) -- (n1_tr);
                \draw[->] (n4_bot) -- (n1_br);
                \draw[->] (n4_bot) -- (n8);
            \end{scope}

            \begin{scope}[shift={(0,10)}, color=pinkColor]
                \node (z3_top) at (2,2) {$\hbar^4$};
                \node (z2_left) at (1,1) {$\hbar^4$};
                \node (z4_right) at (3,1) {$\hbar^3$};
                \node (z1) at (0,0) {$\hbar^4$};
                \node (z3_mid) at (2,0) {$\hbar^3$};
                \node (z6) at (4,0) {$\hbar^2$};
                \node (z2_left2) at (1,-1) {$\hbar^3$};
                \node (z4_right2) at (3,-1) {$\hbar^2$};
                \node (z3_bot) at (2,-2) {$\hbar^2$};
                \node (z5) at (4,-2) {$\hbar$};
                \node (z4_bot) at (3,-3) {$\hbar$};
                \node (z6_bot) at (4,-4) {$1$};

                \draw[->] (z2_left) -- (z3_top);
                \draw[->] (z4_right) -- (z3_top);
                \draw[->] (z1) -- (z2_left);
                \draw[->] (z3_mid) -- (z2_left);
                \draw[->] (z3_mid) -- (z4_right);
                \draw[->] (z6) -- (z4_right);
                \draw[->] (z2_left2) -- (z1);
                \draw[->] (z2_left2) -- (z3_mid);
                \draw[->] (z4_right2) -- (z3_mid);
                \draw[->] (z4_right2) -- (z6);
                \draw[->] (z3_bot) -- (z2_left2);
                \draw[->] (z3_bot) -- (z4_right2);
                \draw[->] (z5) -- (z4_right2);
                \draw[->] (z4_bot) -- (z3_bot);
                \draw[->] (z4_bot) -- (z5);
                \draw[->] (z6_bot) -- (z4_bot);
            \end{scope}

            \draw[dashed, bend right=30, color=pinkColor] (n1_tr) to (z6_bot);
            \draw[dashed, bend left=30, color=cyanColor] (n1_br) to (z5b);
            \draw[dashed, bend left=50, color=brownColor] (n1_bot) to (z1_l);
        \end{tikzpicture}
    \end{subfigure}
    \caption{The poset of a $D_6$ quiver variety with $\msf v = (3,5,7,8,4,4)$ and $\msf w = (1,0,0,0,1,1)$, elements colored with colors $i$ for $i \in [1, 6]$ on the left, and $\weight_{i,1}(x)$ filled in for each $x \in \mathcal P_p^{i,1}$ for each $i \in \{1,n-1,n\}$ on the right}
    \label{fig:d6-multi-framing}
\end{figure}

\begin{example}
    Consider the $D_6$ quiver variety $\mc M$ with dimension vector $\msf v = (3,5,7,8,4,4)$ and framing $\msf w = (1,0,0,0,1,1)$. From Proposition~\ref{prop:mq-dim}, we see that $\dim_\CC \mc M = 0$, and $\mc M^{\msf T}$ consists of isolated points. Let $p \in \mc M^{\msf T}$. From the description of the torus fixed points and the poset $\mathcal P_p$ above, we notice that $\mathcal P_p$ splits into a disjoint union of the three posets $\mathcal P_p^{1,1}$, $\mathcal P_p^{n-1,1}$, and $\mathcal P_p^{n,1}$. Then the dimension vector $\msf v$ also respectively splits into the sum of following three vectors: $\msf v^{1,1} = (1,2,2,2,1,1)$, $\msf v^{n-1,1} = (1,2,3,3,2,1)$, and $\msf v^{n,1} = (1,1,2,3,1,2)$, whose posets are illustrated in Figure \ref{fig:d6-multi-framing}.
    One can compute the tautological bundle weights from Equation~\ref{tbweights} explicitly as follows:
    \begin{itemize}
        \item $\tb_1|_p = a_{1,1} + a_{n-1,1} \hbar^4 + a_{n,1} \hbar^4$,
        \item $\tb_2|_p = a_{1,1} (1+\hbar^4) + a_{n-1,1} (\hbar^3 + \hbar^4) + a_{n,1} \hbar^3$,
        \item $\tb_3|_p = a_{1,1} (1+\hbar^3) + a_{n-1,1} (\hbar^2 + \hbar^3 + \hbar^4) + a_{n,1} (\hbar^2 + \hbar^3)$,
        \item $\tb_4|_p = a_{1,1} (1+\hbar^2) + a_{n-1,1} (\hbar + \hbar^2 + \hbar^3) + a_{n,1} (\hbar + \hbar^2 + \hbar^3)$,
        \item $\tb_5|_p = a_{1,1} + a_{n-1,1} \hbar + a_{n,1} (1+\hbar^2)$,
        \item $\tb_6|_p = a_{1,1} + a_{n-1,1} (1+\hbar^2) + a_{n,1} \hbar$.
    \end{itemize}
\end{example}

\begin{definition}
    A reverse plane partition over a poset $\pst$ is a function $\pi:\pst \to \mathbb{N}$ such that $x \leq y \implies f(y) \leq f(x)$. If $\pst$ is colored by a coloring function $c:\pst \to I$, then we define the degree of a reverse plane partition $\pi$ to be
    \[
        \deg(\pi):=\left(\sum_{x \in c^{-1}(i)} \pi(x) \right)_{i \in I} \in \NN^{I}.
    \]
\end{definition}
Let $\rpp(\pst)$ be the set of reverse plane partitions over $\pst$ and, if a coloring of $\pst$ is provided, let $\rpp^{\bs d}(\pst)$ be the set of degree $\bs d$ reverse plane partitions over $\pst$.
\begin{proposition}\label{fixedqm}
    Let $p \in \mc{M}^{\msf{T}}$. Then
    \[
        \left(\msf{QM}_{\text{ns}, \infty}^{\bs d}\right)^{\msf{T} \times \CC^{\times}_{q}}= \bigoplus_{p \in \mc{M}^{\msf{T}}} \left(\msf{QM}_{\text{ns},\infty,p}^{\bs d}\right)^{\msf{T}\times \mathbb{C}^{\times}_{q}}
    \]
    where
    \[
        \msf{QM}_{\text{ns},\infty,p}^{\bs d}=\{f \in \msf{QM}_{\text{ns},\infty}^{\bs d} \, \mid \, f(\infty)=p\}.
    \]
    Furthermore, $\left(\msf{QM}_{\text{ns},\infty,p}^{\bs d}\right)^{\msf{T}\times \mathbb{C}^{\times}_{q}}$ is in canonical bijection with $\rpp^{\bs d}(\pst_{p})$.
\end{proposition}
\begin{proof}
    The first claim follows from the fact that $\infty$ is a nonsingular point fixed by $\CC^{\times}_{q}$. The second claim is standard, and the argument is identical to that in Section~5.1.3 of \cite{dinkinsthesis}.
\end{proof}

The class $\tau$ can be viewed as a function of variables $x_{i,j}$ for $i \in Q_0$ and $1 \leq j \leq \msf{w}_{i}$ which is symmetric in $x_{i,1},\ldots,x_{i,\msf{w}_{i}}$ for each $i$ separately. Choose some bijection between $\{x_{i,j}\}_{1 \leq j \leq \msf{w}_{i}}$ and $\col_{p}^{-1}(i)$ for each $i$ so that we can view $\tau$ as a function of the elements of $\pst_{p}$. For $\pi \in \rpp(\pst_{p})$, let $\tau(\pi):=\tau|_{x=\weight(x) q^{\pi(x)}}$.

This is sufficient information to compute the vertex function $V^{(\tau)}|_{p}$ by localization. The proof is analogous to the computations of \cite{dinkinssmirnov2020quasimaps} and \cite{psz2020baxterquantum} since the only type $D$ specific ingredients are given by the poset arising in Proposition~\ref{fixedqm}. See also Proposition~1 of \cite{afo2018quantumlanglands}.

\begin{theorem}\label{verformula}
    The restriction of the vertex function of $\mc{M}$ to a point $p \in \mc{M}^{\msf{T}}$ is given by the following formula:
    \begin{multline*}
        V^{(\tau)}(\bs z)|_{p}=\sum_{\pi \in \rpp(\pst_{p})} \left(-\frac{q}{\hbar^{1/2}}\right)^{N(\pi)} \tau(\pi) \bs{z}^{\deg (\pi)}   \\
        \cdot  \left(\prod_{i \in Q_0} \prod_{\substack{x \in \col_{p}^{-1}(i) \\ 1 \leq j \leq \msf{w}_{i}}} \frac{\left( \hbar \frac{\weight(x)}{a_{i,j}}\right)_{\pi(x)}}{\left( q \frac{\weight(x)}{a_{i,j}}\right)_{\pi(x)}} \right)  \left( \prod_{e \in Q_1} \prod_{\substack{ x \in \col_{p}^{-1}(o(e)) \\ y \in \col_{p}^{-1}(i(e))}} \frac{\left(\hbar \frac{\weight(y)}{\weight(x)}\right)_{\pi(y)-\pi(x)}}{{\left(q \frac{\weight(y)}{\weight(x)}\right)_{\pi(y)-\pi(x)}}}\right) \\
        \cdot \left(\prod_{i \in Q_0} \prod_{x,y \in \col_{p}^{-1}(i)} \frac{\left(q \frac{\weight(y)}{\weight(x)}\right)_{\pi(y)-\pi(x)}}{\left(\hbar \frac{\weight(y)}{\weight(x)}\right)_{\pi(y)-\pi(x)}} \right),
    \end{multline*}
    where $N(\pi)=\sum_{e \in Q_{1}} \left(\msf{v}_{o(e)} \deg(\pi)_{i(e)}- \msf{v}_{i(e)} \deg(\pi)_{o(e)}\right) + \sum_{i \in Q_{0}} \msf{w}_{i} \deg(\pi)_{i}$.
\end{theorem}

    The quantity $N(\pi)$ can also be written as
    \[
        N(\pi)=\sum_{j \in Q_{0}} b_{j} \deg(\pi)_{j}.
    \]
    where
    \[
b_j=\left(\sum_{\substack{e \in Q_{1} \\ i(e)=j}} \msf{v}_{o(e)}- \sum_{\substack{e \in Q_{1} \\ o(e)=j}} \msf{v}_{i(e)} +\msf{w}_{j}\right)
    \]
    
We define $\msV^{(\tau)}(\bs z)=V^{(\tau)}(\bs z)|_{z_{i}=(-\hbar^{-1/2})^{b_i} z_{i}}$. This is the same shift appearing in the pole subtraction property of \cite{aganagic2021elliptic} and the 3d mirror symmetry of vertex functions of \cite{bottadink}.

\begin{example}
    For concreteness, consider the simplest example of $Q = D_1$, $\msf v = (1)$, $\msf w = (1)$. For this example, we have
    \[
        \msV^{(1)}(\bs z) = \sum_{k \geq 0} \frac{(\hbar)_k}{(q)_k} \left(\frac{q}{\hbar} z_1\right)^k = F\left(\frac{q}{\hbar} z_1\right),
    \]
    which is a product of $q$-binomial series.
\end{example}

While Theorem~\ref{verformula} is entirely algorithmic, it is combinatorially complicated. Nevertheless, we will show that there is a beautiful summation formula for it, providing a vast generalization of the $q$-binomial theorem, when $|\msf{w}|=1$.

\subsection{Quasimaps to type \texorpdfstring{$D$}{D} points}

Now assume further that $\msf{w}=\delta_{k}$ where $k \in \{1,n-1,n\}$. Let $\mathsf{v}$ be such that the quiver variety $\mc M$ is nonempty, which by Proposition~\ref{minframing} means that it must be a single point. Furthermore, the torus acting on $\mc{M}$ is $\CC^{\times} \times \CC^{\times}_{\hbar}$. Since the diagonal subtorus of the framing torus always acts trivially, we can ignore the first factor and define $\msf T = \CC_{\hbar}^\times$.

So $K_{\msf{T}}(\mc{M})=K_{\msf{T}}(pt)=\ZZ[\hbar^{\pm 1}]$, and hence the vertex function is an element of $K_{\msf{T} \times \CC^{\times}_{q}}(\mc{M})_{loc}[[z]]=\QQ(q,\hbar)[[z_{i}]]_{i \in Q_0}$.

The moduli space $\msf{QM}_{\mathrm{ns},\infty}^{\bs d}$ was defined for arbitrary Nakajima quiver varieties in Section~\ref{subsec:qmdef}. Intuitively, $\msf{QM}_{\mathrm{ns},\infty}^{\bs d}$ is supposed to be a moduli space of maps from $\PP^{1}$ to the quiver variety $\mc{M}$. Nevertheless, even in the case when $\mc{M}$ is a single point, we can spell out the definition of quasimaps to see that $\msf{QM}_{\mathrm{ns},\infty}^{\bs d}$ is a nontrivial space.
\begin{proposition}
    Suppose that $\msf{w}=\delta_{k}$ where $k \in \{1,n-1,n\}$. The moduli space $\msf{QM}_{\text{ns},\infty}^{\bs d}$ is the stack classifying the following data:
    \begin{itemize}
        \item rank $\msf v_i$ vector bundles $\mc{V}_i$ over $\PP^1$ of degrees $\deg(\mc{V}_i)=d_i$ for $i \in Q_0$,
        \item a stable section $s$ of the vector bundle $\mathscr{P} \oplus \mathscr{P}^*$ with
              \[
                  \mathscr{P} = \mc{V}_k \oplus \left(\bigoplus_{i \leq n-2} \Hom\left(\mc{V}_i, \mc{V}_{i+1}\right)\right) \oplus \Hom\left(\mc{V}_{n-2}, \mc{V}_n\right),
              \]
    \end{itemize}
    which is non-singular at $\infty \in \mathbb{P}^1$ and satisfies the moment map equations.
\end{proposition}

In the previous proposition, stability of the section $s$ means the following. For each $p \in \mathbb{P}^{1}$, one can view $s(p)$ as an element of $T^{*}\Rep_{Q}(\msf{v},\msf{w})$. Then the section $s$ is stable if and only if $s(p)$ is stable in the sense of Proposition~\ref{stability} for all but finitely many $p$.

Thus in this case, the vertex function $V^{(\tau)}(\bs z)$ is the generating function for the equivariant Euler characteristics of certain $K$-theory classes on an interesting moduli space.

\section{Vertex functions for zero-dimensional varieties}\label{sec:main-thm}

Our main result is a product formula for the vertex function for zero-dimensional type $D$ Nakajima quiver varieties with one framing at a minuscule vertex. To write our formula, we need some notation.

Fix the quiver $Q=D_{n}$. The associated root system can be constructed as follows. Let $\eps_{i}$ for $1 \leq i \leq n$ be an orthonormal basis of an $n$-dimensional Euclidean space. The simple roots of the $D_{n}$ root system are constructed as $\alpha_{i}=\eps_{i}-\eps_{i+1}$ for $1 \leq i \leq n-1$ and $\alpha_{n}=\eps_{n-1}+\eps_{n}$. The set of all roots is $\Phi=\{\epsilon_{i}\pm \epsilon_{j} \, \mid \, 1 \leq i < j \leq n\}$. It is a disjoint union of positive and negative roots: $\Phi=\Phi^{+} \sqcup \Phi^{-}$.

Fix $\msf{v}$ and $\msf{w}$ such that $\msf{w}_{i}=1$ for a single minuscule vertex $i$ and $\msf{w}_{j}=0$ for $j\neq i$. Assume that the quiver variety $\mc{M}(\msf{v},\msf{w})$ is nonempty.

We identify the group algebra of the root lattice with the K\"ahler variables by the rule
\[
e^{\alpha_{j}}=\left(\frac{q}{\hbar}\right)^{a_{j}} z_{j}, \quad a_j=\msf{v}_{j}-\sum_{\substack{e \\ o(e)=j}} \msf{v}_{i(e)}.
\]
Note that $a_i$ is the $i$th component of $\msf{v}- A_{Q} \msf{v}$, where $A_{Q}$ is the adjacency matrix of the $D_{n}$ quiver for our chosen orientation. Let $\omega_{i}$ for $1 \leq i \leq n$ be the fundamental weights.

As is standard when working with quiver varieties, see \cite{nakajima1994instantons}, we view $\msf{w}$ and $\msf{v}$ as encoding two weights, $\lambda$ and $\mu$, by the rules $\lambda=\sum_{i} \msf{w}_{i} \omega_{i}$ and $\mu=\lambda-\sum_{i} \msf{v}_{i} \alpha_{i}$.

\begin{theorem}\label{thm:product-identity-d}
   The vertex function $\msV(\bs z)$ of $\mc{M}(\msf v, \msf w)$ satisfies the following identity:
    \[
        \msV(\bs z) = \prod_{\alpha \in \Phi_{\mu}^{+}} F(e^{\alpha}),
    \]
    where $\Phi_{\mu}^{+}= \{\alpha \in \Phi^{+} \mid (\mu, \alpha) < 0\}$.
\end{theorem}

We will prove the theorem by considering the cases $\msf{w}=\delta_{i}$ for $i \in \{1,n-1,n\}$ separately. By symmetry, the case of $\msf{w}=\delta_{n-1}$ is identical to the case of $\msf{w}=\delta_{n}$, and so we will only consider the latter below. The power for the proof comes from the Cauchy identity for Macdonald polynomials. 

\subsection{Relation with 3d mirror symmetry} 

Theorem~\ref{thm:product-identity-d} can be viewed as a special case of 3d mirror symmetry for vertex functions. In general, 3d mirror symmetry relates vertex functions of a variety with those of its 3d mirror dual. For a precise statement when all objects are defined, see \cite{bottadink}. 

Let $\mc{M}^{!}$ denote the variety 3d mirror dual to $\mc{M}(\msf{v},\msf{w})$. It can be constructed as a resolution of a slice in an affine Grassmannian. It is shown in \cite{krylovperunov2021almost} that it is isomorphic to an affine space equipped with a torus action such that $0$ is the unique fixed point. Since quasimaps are only defined when the target is presented as a GIT quotient, there is at present no known definition of vertex functions of $\mc{M}^{!}$. Nevertheless, since $\mc{M}(\msf{v},\msf{w})$ is zero-dimensional, it is expected that the vertex function of $\mc{M}^{!}$ is trivial. This is a special case of Conjecture~1 of \cite{dinkinssmirnov2020characters}, which claims that the vertex functions of zero-dimensional quiver varieties factorize into a product of $q$-binomials, each of which corresponds to a repelling weight space in the tangent space $T_{0} \mc{M}^{!}$.

Remark~5.15 of \cite{krylovperunov2021almost} establishes that set $\Phi_{\mu}^{+}$ is exactly the set of repelling weights of $T_0 \mc M^!$. Thus the factorization in Theorem~\ref{thm:product-identity-d} can also be viewed as a proof of the $3d$ mirror symmetry for vertex functions of these type $D$ Nakajima varieties. 

Furthermore, the results of \cite{krylovperunov2021almost} combined with \cite{okounkov2017lectures} Lemma 7.3.5 confirms Conjecture~1 of \cite{dinkinssmirnov2020characters} for any type $D$ quiver variety with framings only at minuscule nodes. See \cite{quasimaphikita} for more details.

\begin{remark}
    Theorem~\ref{thm:product-identity-d}, along with our method of proof, allows one to write explicit formulas for $V^{(\tau)}(\bs z)$, the vertex with descendant $\tau$, whenever $\tau$ is in the subalgebra generated by exterior powers of tautological vector bundles and their duals. In fact, Lemmas \ref{lem:lhs-identification}, \ref{verequalsmac}, and \ref{verequalsmac2} below express the vertex functions in terms of Macdonald polynomials. One immediately sees that applying various Macdonald difference operators to the expressions in these lemmas gives vertex functions with these descendants. Applying the same difference operator to the right hand side of Theorem~\ref{thm:product-identity-d} and using the identity $F(x q)=\left(\frac{1-x}{1-\hbar x}\right)F(x)$ then gives an explicit rational formula for $V^{(\tau)}(\bs z)/V^{(1)}(\bs z)$. Up to global scalar, this ratio is exactly the capped vertex function with descendant. See \cite{dinkinssmirnov2020capped} where this was carried out in the type $A$ setting.
\end{remark}

Let us conclude this section with two examples of the main theorem.

\begin{example}
    The smallest nontrivial example of Theorem~\ref{thm:product-identity-d} is $Q=D_{4}$ with $\msf v = (1,2,1,1)$ and $\msf w = (0,0,0,1)$. Using the localization formula and a computer, one can see that
    \begin{equation}\label{thmexample}
        \msV(\bs z) = F( z_2) \cdot F( (q/
        \hbar) z_2 z_3) \cdot F((\hbar/q) z_1 z_2) \cdot F( z_1 z_2 z_3) \cdot F((q/\hbar) z_1 z_2^2 z_3 z_4).
    \end{equation}
    Even in this example, this identity is nontrivial; in particular, there is no bijection matching reverse plane partitions over the corresponding poset with the terms in the summation on the right hand side.
    To match the root-theoretic description, we calculate
    \[\mu = \omega_4 - (\alpha_1 + 2 \alpha_2 + \alpha_3 + \alpha_4) =\omega_4 - \omega_2.
    \]
    Using the fact that $(\omega_i, \alpha_j) = \delta_{ij}$, one can compute $(\mu, \alpha) = (\omega_4 - \omega_2, \alpha)$ for each $\alpha \in \Phi$. We find that
    \begin{align*}
        (\mu, \alpha_2)=-1,                                \\
        (\mu, \alpha_2+\alpha_3)=-1,                       \\
        (\mu, \alpha_1+\alpha_2)=-1,                       \\
        (\mu, \alpha_1+\alpha_2+\alpha_3)=-1,              \\
        (\mu, \alpha_1+2 \alpha_2+\alpha_3+\alpha_{4})=-1,
    \end{align*}
    and $(\mu, \alpha)\geq 0$ otherwise. Furthermore, $\msf{v}-A_{Q} \msf{v}=(-1,0,1,1)$. This confirms Theorem~\ref{thm:product-identity-d}.
\end{example}

\begin{example}
    Let us write a nontrivial example of Theorem~\ref{thm:product-identity-d} when the framing is at the first node. For example, take $Q=D_5$, $\msf v = (2,2,2,1,1)$, $\msf{w}=(1,0,0,0,0)$. Again using the localization formula, one can experimentally find that
    \begin{align} \nonumber
        \msV(\bs z) & = F( z_1) F( z_1 z_2) F(z_1 z_2 z_3) F((q/\hbar)z_1 z_2 z_3 z_4) F((q/
        \hbar)z_1 z_2 z_3 z_5)    \\
                 & F((q/\hbar)^{2} z_1 z_2 z_3 z_4 z_5) F((q/\hbar)^{2} z_1 z_2 z_3^2 z_4 z_5) F((q/\hbar)^{2} z_1 z_2^2 z_3^2 z_4 z_5),
        \label{thmexample2}
    \end{align}
    To match with the root-theoretic description, we calculate
    \[
        \mu = \omega_1 - (2 \alpha_1 + 2 \alpha_2 + 2 \alpha_3 + \alpha_4 + \alpha_5) =-\omega_1.
    \]
    We find that
    \begin{align*}
        (\mu,\alpha_1)=-1,                                               \\
        (\mu,\alpha_1+\alpha_2)=-1,                                      \\
        (\mu,\alpha_1 + \alpha_2 + \alpha_3)=-1,                         \\
        (\mu,\alpha_1 + \alpha_2 + \alpha_3 + \alpha_4)=-1,              \\
        (\mu,\alpha_1 + \alpha_2 + \alpha_3 + \alpha_5)=-1,              \\
        (\mu,\alpha_1 + \alpha_2 + \alpha_3 + \alpha_4 + \alpha_5)=-1,   \\
        (\mu,\alpha_1 + \alpha_2 + 2 \alpha_3 + \alpha_4 + \alpha_5)=-1, \\
        (\mu,\alpha_1 + 2 \alpha_2 + 2 \alpha_3 + \alpha_4 + \alpha_5)=-1,
    \end{align*}
    and that $(\mu, \alpha)\geq 0$ otherwise. Furthermore, $\msf{v}-A_{Q} \msf{v}=(0,0,0,1,1)$. This confirms Theorem~\ref{thm:product-identity-d} in this case.
\end{example}

\section{Proof of Theorem~\ref{thm:product-identity-d}, part 1}\label{sec:thm-proof1}

In this section, we prove Theorem~\ref{thm:product-identity-d}.

\subsection{Identification with half space Macdonald process}\label{sec: macdonald process}

Assume that $\msf{v}$ and $\msf{w}$ are chosen as in Theorem \ref{thm:product-identity-d} and that $\msf{w}_{n}=1$. Using Theorem~\ref{verformula}, we have
\begin{equation}\label{eq:super-explicit-v}
    \begin{split}
        \msV(\bs z) & = \sum_{\bs{d}} \left(\frac{q}{\hbar}\right)^{N(\bs d)}\bs{z}^{\bs{d}} \left(\prod_{j=1}^{\msf v_n} \frac{(\hbar^{2j-1})_{d_{n,j}}}{(q \hbar^{2j-2})_{d_{n,j}}}\right)       \left(\prod_{i=1}^{n-3} \prod_{j=1}^{\msf v_i} \prod_{k=1}^{\msf v_{i+1}} \frac{\left(\hbar^{k-j}\right)_{d_{i+1,k} - d_{i,j}}}{\left(q \hbar^{k-j-1}\right)_{d_{i+1,k} - d_{i,j}}}\right)                       \\ &
        \left(\prod_{j=1}^{\msf v_{n-2}} \prod_{k=1}^{\msf v_{n-1}} \frac{\left(\hbar^{2k-j}\right)_{d_{n-1,k} - d_{n-2,j}}}{\left(q \hbar^{2k-1-j}\right)_{d_{n-1,k} - d_{n-2,j}}}\right) \left(\prod_{j=1}^{\msf v_{n-2}} \prod_{k=1}^{\msf v_{n}} \frac{\left(\hbar^{2k-1-j}\right)_{d_{n,k} - d_{n-2,j}}}{\left(q \hbar^{2k-j-2}\right)_{d_{n,k} - d_{n-2,j}}}\right) \\ &
        \left(\prod_{i=1}^{n-2} \prod_{j=1}^{\msf v_i} \prod_{k=1}^{\msf v_i} \frac{\left(q \hbar^{k-j}\right)_{d_{i,k} - d_{i,j}}}{\left(\hbar^{k-j+1}\right)_{d_{i,k} - d_{i,j}}}\right)
        \left(\prod_{j=1}^{\msf v_{n-1}} \prod_{k=1}^{\msf v_{n-1}} \frac{\left(q \hbar^{2(k-j)}\right)_{d_{n-1,k} - d_{n-1,j}}}{\left(\hbar^{2(k-j)+1}\right)_{d_{n-1,k} - d_{n-1,j}}}\right)                                                                                                                                                                            \\ & \left(\prod_{j=1}^{\msf v_n} \prod_{k=1}^{\msf v_n} \frac{\left(q \hbar^{2(k-j)}\right)_{d_{n,k} - d_{n,j}}}{\left(\hbar^{2(k-j)+1}\right)_{d_{n,k} - d_{n,j}}}\right),  \\
    \end{split}
\end{equation}
where the multi-index $\bs d$ gives a reverse plane partition over the corresponding poset. We will first rewrite this as a sum over certain tuples of interlacing partitions. 

Define integers $l_{0}=0,l_1,\ldots,l_{n-1},l_{n}=0$ by
\[
l_i = \begin{cases}
    \msf{v}_{i} & \text{if $1\leq i \leq n-2$} \\
    \msf{v}_{n-2}+\msf{v}_{n-1} & \text{if $i=n-1$}
\end{cases}
\]
To express the interlacing relations, we define a sign function as follows.
\begin{definition}\label{def:sign-fn}
    The sign function $\tau : \{1,2,\ldots ,n\} \to \{1,-1\}$ is defined by
    \[
        \tau(i):=(-1)^{l_{i}-l_{i-1}-1} 
    \]
\end{definition}

\begin{definition}
We say that a tuple of partitions $\bs{\lambda} := \left(\lambda^{0}=\emptyset,\lambda^{1}, \ldots, \lambda^{n-1}\right)$ interlaces according to $\msf v$ if the following holds:
\begin{itemize}
    \item $\tau(i) = 1 \implies \lambda^{i} \succ \lambda^{i-1}$ for $1 \leq i \leq n-1$,
    \item $\tau(i) = -1 \implies \lambda^{i} \prec \lambda^{i-1}$ for $1 \leq i \leq n-1$,
    \item $l(\lambda^{i}) \leq l_i$ for $1 \leq i \leq n-1$
\end{itemize}
\end{definition}
We write $S_{\msf{v}}$ for the set of all tuples of partitions that interlace according to $\msf v$. It will be convenient to use the convention that $\lambda_{i}=0$ for $i>l(\lambda)$.

\begin{proposition}\label{prop:vertex-fn-lambda}
    We can rewrite the vertex function in (\ref{eq:super-explicit-v}) as
    \begin{equation}\label{eq:lambda-v}
        \begin{split}
        \msV(\bs z) & = \sum_{\bs{\lambda} \in S_{\msf{v}}} \alpha_{\lambda^{n-1}} \left(\prod_{i=1}^{n-2} \beta_{\lambda^{i}, \lambda^{i+1}}\right) \left(\prod_{i=1}^{n-1} \gamma_{\lambda^{i}}\right) {\bs z}^{\bs \lambda},
        \end{split}
    \end{equation}
    where 
    \begin{align}
            \alpha_{\lambda^{n-1}}             & = \prod_{2 \nmid j}^{l_{n-1}} \frac{(\hbar^{l_{n-1}-j+1})_{\lambda_j^{n-1}}}{(q\hbar^{l_{n-1}-j})_{\lambda_j^{n-1}}},     \label{eq:alpha}                                                                                                                   \\
            \beta_{\lambda^{i}, \lambda^{i+1}} & = \prod_{j=1}^{l_i} \prod_{k=1}^{l_{i+1}} \frac{(\hbar^{l_{i+1}-l_{i}-k+j})_{\lambda_k^{i+1}-\lambda_j^{i}}}{(q\hbar^{l_{i+1}-l_{i}-k+j-1})_{\lambda_k^{i+1}-\lambda_j^{i}}} \text{ for $1\leq i \leq n-2$}, \label{eq:beta} \\
            \gamma_{\lambda^{i}}               & = \begin{cases}
                                                       \prod_{j,k=1}^{l_i} \frac{(q\hbar^{j-k})_{\lambda_k^{i}-\lambda_j^{i}}}{(\hbar^{j-k+1})_{\lambda_k^{i}-\lambda_j^{i}}}        & \text{ if $1 \leq i \leq n-2$}, \\
                                                       \prod_{2 \mid j-k}^{l_{n-1}} \frac{(q\hbar^{j-k})_{\lambda_k^{n-1}-\lambda_j^{n-1}}}{(\hbar^{j-k+1})_{\lambda_k^{n-1}-\lambda_j^{n-1}}} & \text{ if $i = n-1$},
                                                   \end{cases}     \label{eq:gamma}
    \end{align}
    and
    \begin{equation*}
        \bs{z}^{\bs{\lambda}} := (q/\hbar)^{N(\bs \lambda)} \left(\prod_{i=1}^{n-2} z_i^{|\lambda^i|}\right) z_{n-1}^{|\lambda^{n-1,e}|} z_n^{|\lambda^{n-1,o}|}.
    \end{equation*}
    Here the ``even" and ``odd" parts of $\lambda^{n-1}$ are defined by $\lambda^{n-1,o}=(\ldots, \lambda^{n-1}_{l_{n-1}-2},\lambda^{n-1}_{l_{n-1}})$ and $\lambda^{n-1,e}=(\ldots, \lambda^{n-1}_{l_{n-1}-3},\lambda^{n-1}_{l_{n-1}-1})$
\end{proposition}

\begin{proof}
    Given $\bs d$, we define $\bs \lambda$ by $\lambda^{i} = (d_{i,\msf v_i}, d_{i,\msf v_i - 1}, \dots, d_{i,1})$ for each $i = 1, 2, \dots, n-2$ and $\lambda^{n-1} = (\dots,d_{n-1,2}, d_{n,2} ,d_{n-1,1}, d_{n,1})$. Now the result follows by simply regrouping the terms of \eqref{eq:super-explicit-v}.
\end{proof}

We next interpret the right-hand side of (\ref{eq:lambda-v}) in terms of Macdonald polynomials. Fix variable sets $\rho_{i}$ for $0 \leq i \leq n-1$ For $\bs{\lambda} \in S_{\msf{v}}$ and $0 \leq i \leq n-1$, we define
\[
 \mc{W}(\bs \lambda; i) := \begin{cases}
        \mc{E}_{\lambda^{n-1}}(\rho_i)    & \text { if } i = n, \\
        Q_{\lambda^{i-1} / \lambda^{i}}(\rho_i) & \text { if } \tau(i) = -1 \text{ and } 1\leq i \leq n-1,   \\
        P_{\lambda^{i} / \lambda^{i-1}}(\rho_i) & \text { if } \tau(i) =1 \text{ and } 1\leq i \leq n-1,
    \end{cases}
\]
and let
\[
    \mc{W}(\bs{\lambda}) = \prod_{i=1}^{n} \mc{W}(\bs \lambda; i).
\]
Here, the notations $\mc{E}$, $Q$, and $P$ were defined in Section~\ref{sec:macdonald}.


Comparing with \cite{bbc2020halfspacemacdonald}, one sees that the weight $\mc{\bs \lambda}$ is exactly the weight used in the definition of half space Macdonald processes.


We assume that all the variable sets consist of a single variable, i.e. that $\rho_{i}=(x_{i},0,0,\ldots)$. In this case, $P_{\lambda/\mu}$, $Q_{\lambda/\mu}$, and $\mc{E}_{\lambda}$ become
\begin{align*}
    &P_{\lambda/\mu}(x)=
    \begin{cases}
        \psi_{\lambda/\mu} x_1^{|\lambda|-|\mu|} & \text{ if $\lambda \succ \mu$}, \\
        0                                        & \text{ otherwise},
    \end{cases} \\
&    Q_{\lambda/\mu}(x)=
    \begin{cases}
        \varphi_{\lambda/\mu} x_1^{|\lambda|-|\mu|} & \text{ if $\lambda \succ \mu$}, \\
        0                                           & \text{ otherwise}.
    \end{cases} \\
    & \mc{E}_\lambda(x) = b_{e(\lambda)}^{\text{el}} \psi_{e(\lambda)/\lambda} x_1^{|e(\lambda)|-|\lambda|}
\end{align*}
where the notation is as defined in Section~\ref{sec:macdonald}.

Now we identify the coefficients of the vertex function in Proposition~\ref{prop:vertex-fn-lambda} with the weights as follows.
\begin{lemma}\label{lem:lhs-identification}
    Let $\bs{\lambda} \in S_{\msf{v}}$. After identifying the $x_i$ variables with the $z_i$ variables by 
    \[
         (q/\hbar)^{(\msf{v}-A_{Q} \msf{v})_{i}} z_i = \begin{cases}
            x_i^{\tau(i)} x_{i+1}^{-\tau(i+1)} & \text{ if $1\leq i \leq n-1$},   \\
            x_{n-1}^{\tau(n-1)} x_n^{\tau(n)}          & \text{ if $i = n$},
        \end{cases}
    \]
    we have
    \[
       \mc{W}(\bs \lambda) = \alpha_{\lambda^{n-1}} \left(\prod_{i=1}^{n-2} \beta_{\lambda^{i}, \lambda^{i+1}}\right) \left(\prod_{i=1}^{n-1} \gamma_{\lambda^{i}}\right) {\bs z}^{\bs \lambda},
    \]
\end{lemma}

\begin{proof}
    We expand $\mc{W}(\bs \lambda)$ by writing it as
    \begin{align*}
         &\mc{W}(\bs \lambda) = \left(\prod_{\substack{1 \leq i \leq n-1 \\ \tau(i) = 1}} \psi_{\lambda^{i}/\lambda^{i-1}} x_i^{|\lambda^{i}| - |\lambda^{i-1}|}\right) \left(\prod_{\substack{1 \leq i \leq n-1 \\ \tau(i) = -1}} \varphi_{\lambda^{i-1}/\lambda^{i}} x_i^{|\lambda^{i-1}| - |\lambda^{i}|}\right) \\
         & \left( b_{e(\lambda^{n-1})}^{\text{el}} \psi_{e(\lambda^{n-1})/\lambda^{n-1}} x_{n}^{|e(\lambda^{n-1})|-|\lambda^{n-1}|}\right) 
         \end{align*}
        All the monomials in the $x_i$ are
         \begin{align*}
&\left(\prod_{i=1}^{n-1} x_i^{\tau(i)(|\lambda^{i}|-|\lambda^{i-1}|)}\right) x_{n}^{|e(\lambda^{n-1})|-|\lambda^{n-1}|} \\
&=\left(\prod_{i=1}^{n-2} \left(x_i^{\tau(i)} x_{i+1}^{-\tau(i+1)}\right)^{|\lambda^{i}|}\right) x_{n-1}^{\tau(n-1)|\lambda^{n-1}|} x_{n}^{|e(\lambda^{n-1})|-|\lambda^{n-1}|} \\
&=\left(\prod_{i=1}^{n-2} \left(x_i^{\tau(i)} x_{i+1}^{-\tau(i+1)}\right)^{|\lambda^{i}|}\right) \left(x_{n-1}^{\tau(n-1)} x_{n}^{-\tau(n)}\right)^{|\lambda^{n-1,e}|} \left(x_{n-1}^{\tau(n-1)} x_{n}^{\tau(n)}\right)^{|\lambda^{n-1,o}|}
           \end{align*}
           where the last line follows from the identity $\tau(n)(|\lambda^{n-1,o}|-|\lambda^{n-1,e}|)=|e(\lambda^{n-1})|-|\lambda^{n-1}|$. By our change of variable, this is equal to $(q/\hbar)^{a} \bs z^{\bs \lambda}$ for some $a \in \mathbb{Z}$.

           The remaining terms in $\mc{W}(\bs \lambda)$ contribute
         
         \begin{align*}
         \left(\prod_{i=1}^{n-1} f_i\right) b^{\text{el}}_{e(\lambda^{n-1})} \psi_{e(\lambda^{n-1})/\lambda^{n-1}}
    \end{align*}
   where
    \[
        f_i = \begin{cases}
            \psi_{\lambda^{i}/\lambda^{i-1}}    & \text{ if $\tau(i) = 1$}, \\
            \varphi_{\lambda^{i-1}/\lambda^{i}} & \text{ if $\tau(i) = -1$},
        \end{cases}
    \]
    Thus, we must show
    \begin{align*}
         & \left(\prod_{i=1}^{n-1} f_i\right) b_{e(\lambda^{n-1})}^{\text{el}} \psi_{e(\lambda^{n-1})/\lambda^{n-1}} \left(\frac{q}{\hbar}\right)^{a} \\
         & \quad= \alpha_{\lambda^{n-1}} \left(\prod_{i=1}^{n-2} \beta_{\lambda^{i}, \lambda^{i+1}}\right) \left(\prod_{i=1}^{n-1} \gamma_{\lambda^{i}}\right),
    \end{align*}

    This is a straightforward (though long) computation, and we omit the details.
   
    
\end{proof}

\begin{remark}
    The previous proposition shows that these vertex functions coincide with the partition functions of the half-space Macdonald processes from \cite{bbc2020halfspacemacdonald}.
\end{remark}

In the next subsection, we convert the expression on the right-hand side to a product form.

\subsection{Matching with roots}

As in Proposition~2.2 of \cite{bbc2020halfspacemacdonald}, partition functions for half space Macdonald processes can be summed using Cauchy identities. In our setting, this reads
\begin{equation}
    \msV(\bs z)= \sum_{\bs{\lambda}}\mc{W}(\bs \lambda)
     = \prod_{\substack{1 \leq i < j \leq n \\ \tau(i) = 1}} F(x_i x_{j}). \label{sumisproduct}
\end{equation}
We now prove that terms in the product match certain roots. Taking into account our identification of the $x$ and $z$ variables, we have the following.

\begin{lemma}\label{lem:roots-identity}
   We have the equality of sets 
    \begin{align*}
        \{\tau(i) \epsilon_{i}+\tau(j) \epsilon_{j} : 1 \leq i < j \leq n, \tau(i) = 1\} = \Phi_{\mu}^{+}
    \end{align*}
\end{lemma}

\begin{proof}
By definition,
\[
\mu=\omega_{n}-\sum_{i=1}^{n} \msf{v}_{i} \alpha_{i}=\sum_{i=1}^{n-1}\left(\frac{1}{2}-l_{i}+l_{i-1}\right)\epsilon_{i}+\left(\frac{1}{2} -\msf{v}_{n}+\msf{v}_{n-1}\right)\epsilon_{n}
\]
Write the above expression as $\mu=\sum_{i=1}^{n} a_{i} \epsilon_{i}$. We note that $a_{i} \in \{1/2,-1/2\}$ for all $i$, and $a_{i}=1/2 \iff \tau(i)=-1$. 

Let $\alpha\in \Phi^{+}$, i.e., $\alpha=\epsilon_{i} + s \epsilon_{j}$ where $i<j$ and $s\in \{-1,1\}$. Then $(\alpha,\mu)=a_{i}+s a_{j}$. So 
\[
(\alpha,\mu)<0 \iff a_{i}=-\frac{1}{2} \text{ and } s a_{j}=-\frac{1}{2} \iff \tau(i)=1 \text{ and } \tau(j)=s
\]
This concludes the proof.

\end{proof}

Combining Lemma~\ref{lem:lhs-identification}, Equation~\eqref{sumisproduct}, and Lemma~\ref{lem:roots-identity} concludes the proof of Theorem~\ref{thm:product-identity-d} in the case where the framing is at a spin node.

\section{Proof of Theorem~\ref{thm:product-identity-d}, part 2}\label{sec:thm-proof2}

Next we prove Theorem~\ref{thm:product-identity-d} when the framing is at the first node: $\msf{w}=(1,0,\ldots,0)$. Despite the simplicity of this setting in the comparison to framing at a spin vertex, the language of half-space Macdonald processes does not apply. So we must develop the relationship with Macdonald polynomials from scratch.

\subsection{Reduction to two cases}

Recall Lemma~\ref{lem:framing-first-configs}, which describes the $\msf{v}$ such that $\mc{M}(\msf{v},\msf{w})$ is nonempty. If $\msf{v}_{i} < 2$ for all $i$, then the quiver varieties is a quotient by an abelian group and Theorem~\ref{thm:product-identity-d} follows from repeated applications of the $q$-binomial theorem. So we can assume that $\msf{v}$ is of the form $\msf v = (1^k, 2^{n-k-2}, 1, 1)$ where $0 \leq k <n-2$. We can further reduce the number of possibilities by the following lemma.

\begin{lemma}\label{delta1-reduction}
    Let $\msf v = (1^k, 2^{n-k-2}, 1, 1)$ where $k \geq 2$. Let $\msf{v}'=(0,1^{k-1},2^{n-k-2},1,1)$ and $\msf{w}'=(0,1,0,\ldots,0)$. If Theorem \ref{thm:product-identity-d} holds for $\mc{M}(\msf{v}',\msf{w}')$, then it holds for $\mc{M}(\msf{v},\msf{w})$.
\end{lemma}

\begin{proof}
Let $\msV(\bs z)=\msV(z_1,\ldots,z_n)$ and $\msV'(\bs z)=\msV'(z_{2},\ldots,z_{n})$ be the vertex functions of $\mc{M}(\msf{v},\msf{w})$ and $\mc{M}(\msf{v}',\msf{w}')$ respectively. Theorem~\ref{verformula} implies that 
\[
\msV(\bs z)= F(z_1 z_2) \msV'(\bs z)
\]

Associated to $\msf{v}$ and $\msf{v}'$ are weights $\mu=\omega_1-\sum_{i=1}^{n} \msf{v}_{i} \alpha_{i}$ and $\mu'=\omega_{2}-\sum_{i=2}^{n} \msf{v}'_{i} \alpha_{i}$, where for the latter we view $D_{n-1} \hookrightarrow D_{n}$.

Applying Theorem \ref{thm:product-identity-d} to $\msV'(\bs z)$ gives
\[
\msV(\bs z)= F(z_1 z_2) \prod_{\alpha \in \Phi_{\mu'}^{+}} F(e^{\alpha})
\]
where we embed $D_{n-1} \hookrightarrow D_{n}$ by forgetting the first node.

 So $\mu-\mu'=\omega_{1}-\omega_{2}-\alpha_{1}=-\epsilon_{1}$. Thus any positive root of $D_{n-1}$ pairing negatively with $\mu'$ will also pair negatively with $\mu$. We also have $\mu=\epsilon_{3}+\sum_{i=3}^{n} \msf{v}_{i} \alpha_{i}$, from which it follows that $(\mu,\alpha_1+\alpha_{2})<0$. This concludes the proof.

\end{proof}

By Lemma~\ref{delta1-reduction}, to prove Theorem~\ref{thm:product-identity-d}, it suffices to prove it when $\msf v=(2^{n-2},1,1)$ and $\msf v =(1,2^{n-3},1,1)$.

\subsection{Case of \texorpdfstring{$\msf v = (2^{n-2}, 1^2)$}{v = (2^{n-2}, 1^2)}}

Assume that $\msf{v}=(2^{n-2},1,1)$. Theorem~\ref{verformula} and Definition~\ref{dfundposet} give the following formula for $\msV(\bs z)$:
\begin{align*}
     & \sum_{\bs d}
    \prod_{i=1}^{n-3}
    \left(\frac{(\hbar)_{d_{i+1,1}-d_{i,1}}}{(q)_{d_{i+1,1}-d_{i,1}}}
    \frac{(\hbar^{n-i-1})_{d_{i+1,2}-d_{i,1}}}{(q \hbar^{n-i-2})_{d_{i+1,2}-d_{i,1}}}
    \frac{(\hbar^{2+i-n})_{d_{i+1,1}-d_{i,2}}}{(q \hbar^{1+i-n})_{d_{i+1,1}-d_{i,2}}}
    \frac{(1)_{d_{i+1,2}-d_{i,2}}}{(q \hbar^{-1})_{d_{i+1,2}-d_{i,2}}}\right)
    \\
     & \cdot
    \frac{(\hbar)_{d_{n-1,1}-d_{n-2,1}}}{(q)_{d_{n-1,1}-d_{n-2,1}}}
    \frac{(1)_{d_{n-1,1}-d_{n-2,2}}}{(q \hbar^{-1})_{d_{n-1,1}-d_{n-2,2}}}
    \frac{(\hbar)_{d_{n,1}-d_{n-2,1}}}{(q)_{d_{n,1}-d_{n-2,1}}}
    \frac{(1)_{d_{n,1}-d_{n-2,2}}}{(q \hbar^{-1})_{d_{n,1}-d_{n-2,2}}}
    \\
     & \cdot
    \frac{(\hbar)_{d_{1,1}}}{(q)_{d_{1,1}}}
    \frac{(\hbar^{n-1})_{d_{1,2}}}{(q \hbar^{n-2})_{d_{1,2}}}
    \left(\prod_{i=1}^{n-2}
    \frac{(q\hbar^{n-i-1})_{d_{i,2}-d_{i,1}}}{(\hbar^{n-i})_{d_{i,2}-d_{i,1}}}
    \frac{(q \hbar^{i+1-n})_{d_{i,1}-d_{i,2}}}{(\hbar^{i+2-n})_{d_{i,1}-d_{i,2}}} \right)
    \left( \frac{q}{\hbar}\right)^{N(\bs d)}\bs z^{\bs d}
\end{align*}
where the multi-index $\bs d$ gives a reverse plane partition over the corresponding poset. Applying \eqref{qpochinv} and changing variables via 
\[
z_{i} = \begin{cases}
     x_1 x_2^{-1} & \text{if $1 \leq i \leq n-2$} \\
     \frac{\hbar}{q} x_{n-1} x_{n}^{-1} & \text{if $i=n-1$} \\
      \frac{\hbar}{q} x_{n-1} x_{n} & \text{if $i=n-1$} 
\end{cases}
\]
we rewrite the above expression as
\begin{align}
    \nonumber & \sum_{\bs d}
    \left(\prod_{i=1}^{n-3}
    \frac{(\hbar)_{d_{i+1,1}-d_{i,1}}}{(q)_{d_{i+1,1}-d_{i,1}}}
    \frac{(\hbar^{n-i-1})_{d_{i+1,2}-d_{i,1}}}{(q \hbar^{n-i-2})_{d_{i+1,2}-d_{i,1}}}
    \frac{(\hbar^{n-i-1})_{d_{i,2}-d_{i+1,1}}}{(q \hbar^{n-i-2})_{d_{i,2}-d_{i+1,1}}}
    \frac{(\hbar)_{d_{i,2}-d_{i+1,2}}}{(q)_{d_{i,2}-d_{i+1,2}}}\right)
    \\
    \nonumber & \cdot
    \frac{(\hbar)_{d_{n-1,1}-d_{n-2,1}}}{(q)_{d_{n-1,1}-d_{n-2,1}}}
    \frac{(\hbar)_{d_{n-2,2}-d_{n-1,1}}}{(q)_{d_{n-2,2}-d_{n-1,1}}}
    \frac{(\hbar)_{d_{n,1}-d_{n-2,1}}}{(q)_{d_{n,1}-d_{n-2,1}}}
    \frac{(\hbar)_{d_{n-2,2}-d_{n,1}}}{(q)_{d_{n-2,2}-d_{n,1}}}                           \\
    \nonumber & \cdot  \frac{(\hbar)_{d_{1,1}}}{(q)_{d_{1,1}}}
    \frac{(\hbar^{n-1})_{d_{1,2}}}{(q \hbar^{n-2})_{d_{1,2}}} \left(
    \prod_{i=1}^{n-2}
    \frac{(q\hbar^{n-i-1})_{d_{i,2}-d_{i,1}}}{(\hbar^{n-i})_{d_{i,2}-d_{i,1}}}
    \frac{(q \hbar^{n-i-2})_{d_{i,2}-d_{i,1}}}{(\hbar^{n-i-1})_{d_{i,2}-d_{i,1}}} \right) \\
              & \cdot  \left(\prod_{\substack{1 \leq i \leq n                             \\ i \neq n-1}}x_i^{d_{i} - d_{i-1}}\right) x_{n-1}^{d_{n-1}+d_{n}-d_{n-2}} \label{verexpanded}
\end{align}
where $d_{i}:=\sum_{j=1}^{\msf{v}_{i}} d_{i,j}$ for $1 \leq i \leq n$ and $d_{0}:=0$. As in the Lemma~\ref{lem:lhs-identification}, we write this in terms of Macdonald polynomials.

\begin{lemma}\label{verequalsmac}
    The vertex function is equal to
    \begin{multline}
        \msV(\bs z) =  \sum \varphi_{\nu^{0}/(a^{n-2})} \left(\prod_{i=2}^{n} x_i \right)^{-a}  \\ Q_{(a)}(x_1) P_{|\nu^{0}|-(n-2)a}(x_1)    \prod_{i=0}^{n-2} P_{\nu^{i}/\nu^{i+1}}(x_{i+2}), \label{vermac}
    \end{multline}
    where the sum is taken over all $n$-tuples of partitions $(\nu^{0}, \nu^{1}, \dots, \nu^{n-1})$ and all $a \in \NN$ such that
    \[
        \emptyset=\nu^{n-1} \prec \nu^{n-2} \prec  \nu^{n-1} \prec \ldots \prec \nu^{0} \succ (a^{n-2}).
    \]
\end{lemma}
\begin{proof}
    Let $\nu^{i}$ for $0\leq i \leq n-1$ and $a \in \NN$ be as in the statement of the lemma. We write them explicitly as
    \begin{itemize}
        \item $\nu^{i} = (\nu^{i}_1, a^{n-3-i}, \nu^{i}_2)$ for $i \in [0, n-3]$,
        \item $\nu^{n-2}=(\nu^{n-2}_{1})$.
    \end{itemize}
    The assignment
    \begin{itemize}
        \item $d_{1,1} = \nu^{0}_{2}$, $d_{1,2} = \nu^{0}_{1}$,
        \item $d_{i,1} = \nu^{0}_{1}+\nu^{0}_{2} - \nu^{i-1}_1$, $d_{i,2} = \nu^{0}_{1}+\nu^{0}_{2}- \nu^{i-1}_2$ for $i \in [2, n-2]$,
        \item $d_{n-1,1} =\nu^{0}_{1}+\nu^{0}_{2}-\nu^{n-2}_1$,
        \item $d_{n,1} = \nu^{0}_{1}+\nu^{0}_{2}-a$
    \end{itemize}
    defines a bijection between the set of $\left((\nu^{i})_{i=0,\ldots,n-1},a\right)$ satisfying the interlacing conditions and reverse plane partitions $\bs d$ over the poset corresponding to $\msf{v}$. It is now straightforward using Equations~\eqref{phipsi} and~\eqref{macnorm} to check that the corresponding terms of Equations~\eqref{verexpanded} and~\eqref{vermac} coincide.
\end{proof}

We can now sum the right hand side of Equation~\eqref{vermac}.

\begin{lemma}
    The vertex function is equal to
    \[
        \msV(\bs z)= \prod_{i=2}^{n} F(x_1 x_i) F(x_1 x_i^{-1}).
    \]
\end{lemma}
\begin{proof}
    We first apply the branching rule \eqref{branching} to get
    \begin{align*}
         & \sum \varphi_{\nu^{0}/(a^{n-2})} \left(\prod_{i=2}^{n} x_i \right)^{-a} Q_{(a)}(x_1) P_{|\nu^{0}|-(n-2)a}(x_1)    \prod_{i=0}^{n-2} P_{\nu^{i}/\nu^{i+1}}(x_{i+2}) \\
         & =\sum \varphi_{\nu^{0}/(a^{n-2})} \left(\prod_{i=2}^{n} x_i \right)^{-a} Q_{(a)}(x_1) P_{|\nu^{0}|-(n-2)a}(x_1)    P_{\nu^{0}}(x_2,\ldots, x_{n})
    \end{align*}
    where the first sum is over the set described in Lemma~\ref{verequalsmac} and second sum is over the set of $a\geq 0$ and $\nu^{0}$ such that $\nu^{0} \succ (a^{n-1})$.
    Applying the Pieri rule \eqref{pieri} gives
    \begin{align*}
         & \sum_{\substack{a \geq 0                                                                                                                        \\ \nu^{0} \succ (a^{n-2})}} \varphi_{\nu^{0}/(a^{n-2})}
        \left(\prod_{i=2}^{n} x_i \right)^{-a} Q_{(a)}(x_1) P_{|\nu^{0}|-(n-2)a}(x_1)   P_{\nu^{0}}(x_2,\ldots, x_{n})                                 \\
         & =\sum_{a,b \geq 0} \left(\prod_{i=2}^{n} x_i \right)^{-a} Q_{(a)}(x_1)  P_{(b)}(x_1)  P_{(a^{n-2})}(x_2,\ldots, x_{n}) Q_{(b)}(x_2,\ldots,x_n).
    \end{align*}
    Applying the inversion identity \eqref{macinv} gives
    \begin{align*}
         & \sum_{a,b \geq 0} \left(\prod_{i=2}^{n} x_i \right)^{-a} Q_{(a)}(x_1)  P_{(a^{n-2})}(x_2,\ldots, x_{n}) P_{(b)}(x_1) Q_{(b)}(x_2,\ldots,x_n) \\
         & =\sum_{a,b \geq 0} Q_{(a)}(x_1) P_{(a)}(x_2^{-1},\ldots, x_{n}^{-1}) P_{(b)}(x_1) Q_{(b)}(x_2,\ldots,x_n).
    \end{align*}
    Now the Cauchy identity finishes the proof.
\end{proof}

To finish the proof of Theorem~\ref{thm:product-identity-d} in this case, it remains to identify terms in the product with certain roots which is the content of the following lemma.

\begin{lemma}
    \[
\Phi^{+}_{\mu}=\{\epsilon_1\pm \epsilon_{i}\, \mid \, 2 \leq i \leq n\}
    \]
\end{lemma}
\begin{proof}
Since $\msf{v}=(2^{n-2},1,1)$, we have $\mu=\omega_{1}-\sum_{i=1}^{n} \msf{v}_{i} \alpha_{i}=-\epsilon_{1}$. The lemma follows.
\end{proof}

\subsection{Case of \texorpdfstring{$\msf v = (1, 2^{n-3}, 1^2)$}{v = (1, 2^{n-3}, 1^2)}}\label{subsec-12n-3}

Assume that $\msf{v}=(1, 2^{n-3}, 1^2)$. Theorem~\ref{verformula} and Definition~\ref{dfundposet} give the following:
\begin{align*}
    \msV(\bs z)
     & = \sum_{\bs d}
    \left(\prod_{i=1}^{n-3}
    \frac{(\hbar)_{d_{i+1,1}-d_{i,1}}}{(q)_{d_{i+1,1}-d_{i,1}}}
    \frac{(\hbar^{n-i-1})_{d_{i+1,2}-d_{i,1}}}{(q \hbar^{n-i-2})_{d_{i+1,2}-d_{i,1}}} \right)
    \\
     & \cdot \left(\prod_{i=2}^{n-3}
    \frac{(\hbar^{2+i-n})_{d_{i+1,1}-d_{i,2}}}{(q \hbar^{1+i-n})_{d_{i+1,1}-d_{i,2}}}
    \frac{(1)_{d_{i+1,2}-d_{i,2}}}{(q \hbar^{-1})_{d_{i+1,2}-d_{i,2}}} \right)
    \frac{(\hbar)_{d_{n-1,1}-d_{n-2,1}}}{(q)_{d_{n-1,1}-d_{n-2,1}}}
    \\
     & \cdot    \frac{(1)_{d_{n-1,1}-d_{n-2,2}}}{(q \hbar^{-1})_{d_{n-1,1}-d_{n-2,2}}} \frac{(\hbar)_{d_{n,1}-d_{n-2,1}}}{(q)_{d_{n,1}-d_{n-2,1}}}
    \frac{(1)_{d_{n,1}-d_{n-2,2}}}{(q \hbar^{-1})_{d_{n,1}-d_{n-2,2}}}
    \frac{(\hbar)_{d_{1,1}}}{(q)_{d_{1,1}}}                                                                                                        \\
     & \cdot \left( \prod_{i=2}^{n-2}
    \frac{(q\hbar^{n-i-1})_{d_{i,2}-d_{i,1}}}{(\hbar^{n-i})_{d_{i,2}-d_{i,1}}}
    \frac{(q \hbar^{i+1-n})_{d_{i,1}-d_{i,2}}}{(\hbar^{i+2-n})_{d_{i,1}-d_{i,2}}} \right)\left(\frac{q}{\hbar}\right)^{N(\bs d)} \bs z^{\bs d}
\end{align*}
where the multi-index $\bs d$ gives a reverse plane partition over the corresponding poset. Applying \eqref{qpochinv} and changing variables via 
\[
z_{i} = \begin{cases}
    \frac{q}{\hbar} x_1 x_2^{-1} & \text{if $i=1$} \\
    x_{i} x_{i+1}^{-1} & \text{if $2 \leq i \leq n-2$} \\
    \frac{\hbar}{q} x_{n-1}x_{n}^{-1} & \text{ if $i=n-1$} \\
    \frac{\hbar}{q} x_{n-1}x_{n} & \text{ if $i=n$}
\end{cases}
\]
we can rewrite this as
\begin{equation}
    \begin{split}\label{verexpanded2}
        V(\bs z) & = \sum_{\bs d}
        \left(\prod_{i=1}^{n-3}
        \frac{(\hbar)_{d_{i+1,1}-d_{i,1}}}{(q)_{d_{i+1,1}-d_{i,1}}}
        \frac{(\hbar^{n-i-1})_{d_{i+1,2}-d_{i,1}}}{(q \hbar^{n-i-2})_{d_{i+1,2}-d_{i,1}}} \right) \\
                 & \cdot \left(
        \prod_{i=2}^{n-3}
        \frac{(\hbar^{n-i-1})_{d_{i,2}-d_{i+1,1}}}{(q \hbar^{n-i-2})_{d_{i,2}-d_{i+1,1}}}
        \frac{(\hbar)_{d_{i,2}-d_{i+1,2}}}{(q)_{d_{i,2}-d_{i+1,2}}}\right)  \frac{(\hbar)_{d_{n-1,1}-d_{n-2,1}}}{(q)_{d_{n-1,1}-d_{n-2,1}}}
        \\
                 & \cdot \frac{(\hbar)_{d_{n-2,2}-d_{n-1,1}}}{(q)_{d_{n-2,2}-d_{n-1,1}}}
        \frac{(\hbar)_{d_{n,1}-d_{n-2,1}}}{(q)_{d_{n,1}-d_{n-2,1}}}
        \frac{(\hbar)_{d_{n-2,2}-d_{n,1}}}{(q)_{d_{n-2,2}-d_{n,1}}}
        \frac{(\hbar)_{d_{1,1}}}{(q)_{d_{1,1}}}
        \\
                 & \cdot
        \left( \prod_{i=2}^{n-2}
        \frac{(q\hbar^{n-i-1})_{d_{i,2}-d_{i,1}}}{(\hbar^{n-i})_{d_{i,2}-d_{i,1}}}
        \frac{(q \hbar^{n-i-2})_{d_{i,2}-d_{i,1}}}{(\hbar^{n-i-1})_{d_{i,2}-d_{i,1}}} \right)     \\
                 & \cdot \left(\prod_{\substack{1 \leq i \leq n                                   \\ i \neq n-1}} x_i^{d_i-d_{i-1}}\right) x_{n-1}^{d_{n-1}+d_{n}-d_{n-2}}
    \end{split}
\end{equation}
where $d_{i}=\sum_{j=1}^{\msf{v}_{i}} d_{i,j}$ and $d_{0}=0$.
As before, we write this in terms of Macdonald polynomials.
\begin{lemma}\label{verequalsmac2}
    The vertex function is equal to
    \begin{multline}
        \msV(\bs z)=\sum \varphi_{\nu^{0}/(a^{n-3})} \left(\prod_{i=3}^{n} x_i \right)^{-a} \\
        Q_{(a)}(x_2) P_{(b)}(x_2) Q_{(b)/(|\nu^{0}|-(n-3)a)}(x_1)  \prod_{i=0}^{n-3} P_{\nu^{i}/\nu^{i+1}}(x_{i+3}), \label{vermac2}
    \end{multline}
    where the sum is taken over the set of all $(n-1)$-tuples of partitions $(\nu^{0},\nu^{1},\ldots,\nu^{n-2})$ and all $a,b \in \NN$ such that
    \[
        \emptyset = \nu^{n-2} \prec \nu^{n-3} \prec \ldots \prec \nu^{0} \succ (a^{n-3}).
    \]
\end{lemma}
\begin{proof}
    Let $a$, $b$, and $\nu^{i}$ be as in the statement of the lemma. We write them explicitly as
    \begin{itemize}
        \item $\nu^{0} = (\nu^{0}_{1}, a^{n-4}, \nu^{0}_2)$,
        \item $\nu^{i} = (\nu^{i}_1, a^{n-4-i}, \nu^{i}_2)$ for $1 \leq i \leq n-4$,
        \item $\nu^{n-3} = (\nu^{n-3}_1)$.
    \end{itemize}
    The assignment
    \begin{itemize}
        \item $d_{1,1} = a+b-\nu^{0}_{1}-\nu^{0}_{2}$,
        \item $d_{2,1} = a+b-\nu^{0}_{1}$, $d_{2,2} = a+b-\nu^{0}_{2}$,
        \item $d_{i,1} = a+b-\nu^{i-2}_1$, $d_{i,2} = a+b - \nu^{i-2}_2$ for $3 \leq i \leq n-2$,
        \item $d_{n-1,1} = a+b-\nu^{n-3}_1$,
        \item $d_{n,1} = a$
    \end{itemize}
    defines a bijection between the set of $\left((\nu^{i})_{i=0,\ldots,n-2},a,b\right)$ satisfying the interlacing conditions and reverse plane partitions $\bs d$ over the poset corresponding to $\msf{v}$. It is now straightforward using Equations~\eqref{phipsi} and~\eqref{macnorm} to check that the corresponding terms of Equations~\eqref{verexpanded2} and~\eqref{vermac2} coincide.
\end{proof}

\begin{lemma}\label{lem:vtxfn-12n-3}
    The vertex function is equal to
    \[
        \msV(\bs z)= \left(\prod_{i \neq 2}^n F(x_2 x_i)\right) \left( \prod_{i =3}^{n} F(x_2 x_i^{-1})\right).
    \]
\end{lemma}
\begin{proof}
    We first apply the branching rule \eqref{branching} to see that the right-hand side of Equation~\eqref{verequalsmac2} is equal to
    \begin{multline*}
        \sum_{\substack{a,b \geq 0 \\ \nu^{0} \succ (a^{n-3})}}   \varphi_{\nu^{0}/(a^{n-3})}
        \left(\prod_{i=3}^{n} x_i \right)^{-a} Q_{(a)}(x_2) P_{(b)}(x_2)  \\ Q_{(b)/(|\nu^{0}|-(n-3)a)}(x_1)  \cdot P_{\nu^{0}}(x_3,\ldots,x_n).
    \end{multline*}
    By the Pieri rule \eqref{pieri}, this is equal to
    \begin{multline*}
        \sum_{a,b,c \geq 0} \left(\prod_{i=3}^{n} x_i \right)^{-a} Q_{(a)}(x_2)  P_{(a^{n-3})}(x_3,\ldots, x_{n}) \\ P_{(b)}(x_2) Q_{(b)/(c)}(x_1) Q_{(c)}(x_3,\ldots,x_n).
    \end{multline*}
    Applying the branching rule, the previous line becomes
    \begin{equation*}
        \sum_{a,b \geq 0} \left(\prod_{i=3}^{n} x_i \right)^{-a} Q_{(a)}(x_2)  P_{(a^{n-3})}(x_3,\ldots, x_{n}) P_{(b)}(x_2) Q_{(b)}(x_1,x_3,\ldots,x_n).
    \end{equation*}
    Applying the inversion identity \eqref{macinv}, this is equal to
    \begin{equation*}
        \sum_{a,b \geq 0} Q_{(a)}(x_2) P_{(a)}(x_3^{-1},\ldots, x_{n}^{-1}) P_{(b)}(x_2) Q_{(b)}(x_1,x_3,\ldots,x_n).
    \end{equation*}
    Now the Cauchy identity finishes the proof.
\end{proof}

To finish the proof of Theorem~\ref{thm:product-identity-d} in this case, it remains to identify terms in the product with certain roots which we do in the following lemma.

\begin{lemma}
    \[
\Phi^{+}_{\mu}=\{\epsilon_{2}+\epsilon_{i} \, \mid \, i \neq 2\} \cup \{\epsilon_{2}-\epsilon_{i} \, \mid \, 3 \leq i \leq n\}
    \]
\end{lemma}
\begin{proof}
    Since $\msf{v}=(1,2^{n-3},1,1)$, we have $\mu=\omega_{1}-\sum_{i=1}^{n} \msf{v}_{i} \alpha_{i}=-\epsilon_{2}$. Clearly, the positive roots in the statement of the lemma are exactly the ones pairing negatively with $\mu$.
\end{proof}


\printbibliography

\end{document}